\documentclass[a4paper, 10pt]{amsart}

\usepackage{amsmath, latexsym,  amssymb, amscd}
\usepackage[all]{xy}
\usepackage[british]{babel}
\usepackage{url}

\usepackage{enumitem}

\usepackage{hyperref}
\usepackage[foot]{amsaddr}

\renewcommand{\epsilon}{\varepsilon}

\DeclareMathOperator{\codim}{codim}

\newcommand{\abs}[1]{\left| #1 \right|}

\newcommand{\rank}{{\rm rank\,}}

\renewcommand{\P}{\mathbb{P}}

\newcommand{\sotto}{B}

\newcommand{\rend}{{\mathrm{End}}(E)}

\newcommand{\PP}{\mathbb{P}}
\newcommand{\qe}{\mathbb{Q}}
\newcommand{\C}{\mathbb{C}}
\newcommand{\R}{\mathbb{R}}
\newcommand{\Z}{\mathbb{Z}}

\newcommand{\Q}{\mathbb{Q}}

\newcommand{\Ci}{\mathcal{C}}

\newcommand{\cquattro}{C}
\newcommand{\ccinque}{C}
\newcommand{\csei}{C_0}
\newcommand{\cotto}{C}

\newtheorem{thm}{Theorem}[section]

\newtheorem{propo}[thm]{Proposition}
\newtheorem{lem}[thm]{Lemma}
\newtheorem{cor}[thm]{Corollary}
\newtheorem{D}[thm]{Definition}

\usepackage{hyperref}

\title[]{An explicit Manin-Dem'janenko theorem in  elliptic curves}

\author{E. Viada $^1$}

\thanks{$^1$ Supported by the FNS Project PP00P2-123262/1}

\begin{document}

%\keywords{Subvarieties of products of elliptic curves, Finiteness of torsion anomalous intersections, Diophantine approximation}
\begin{abstract} 
Let $\Ci$ be a curve of genus at least $2$ embedded in   $E_1 \times \cdots \times E_N$ where the $E_i$ are elliptic curves for $i=1,\dots, N$. In this article we give an explicit sharp bound for the N\'eron-Tate height of the points of  $\Ci$  contained in the union of all algebraic subgroups of dimension    $<\max(r_\Ci-t_\Ci,t_\Ci)$ where   $t_\Ci$,  respectively $r_\Ci$, is the minimal dimension of  a translate, respectively of a torsion variety, containing $\Ci$. 

 As a corollary, we give an explicit bound for the height of the rational points of special curves, proving  new cases of the explicit Mordell Conjecture and in particular making explicit  (and slightly more general in the CM case) the Manin-Dem'janenko method in products of elliptic curves. 
\end{abstract}
\maketitle

\section{Introduction}

By \emph{variety} we mean an algebraic variety defined over the algebraic numbers $\overline\qe$. If not otherwise specified,  we identify  $V$ with its algebraic points $V(\overline{\qe})$.
 We denote by $E$  an elliptic curve and by  $A$  an abelian variety. For a positive integer $N$, we denote by $E^N$ the cartesian product of $N$ copies of $E$. 

We say that a subvariety $V\subset  A$  is a \emph{translate},  respectively  a \emph{torsion variety},  if it is a finite  union of translates of  proper algebraic subgroups  of $A$ by  points, respectively  by torsion points.

Furthermore, an irreducible variety $V\subset  A$  is \emph{transverse}, respectively \emph{weak-transverse}, if  it is not contained in any translate, respectively in any torsion variety.

 In this article the word {\it rank} is used with his several different meanings. For clarity, we remember that the rank of an abelian group is the number of generators over $\Z$ of its free part; the rank of an $R$-module $M$ for $R$ a ring with field of franction ${\rm frac}{(R)}$ is the dimension of the vector space $M\otimes_R {\rm frac}(R)$; the $k$-rank of an abelian variety $A$ defined over $k$, for $k$ a number field, is the rank of $A(k)$ as an abelian group; and the rank of a point on an abelian variety $A$ is the only new concept  introduced in Definition \ref{rank} below. \\

A classical question in the context of Diophantine Geometry is to determine the points on a curve of a certain shape, for instance the  rational points. Much work has been done in this direction.  In a fundamental work Faltings \cite {FaltingsTeo} proved   the Mordell Conjecture, namely  that a curve of genus at least $2$ defined over a number field $k$ has only finitely many $k$-rational points. 

However the innovative  proof of this theorem is not effective in the sense that it does not give any information on how these points could be determined.  By explicit method, we mean a statement that provides an explicit bound for the height of the rational points on a curve, thus also a method to find them. This is one of the consequences  of our paper for curves with a factor of the Jacobian isogenous to a product $E^N$ with $E(\Q)$ of rank at most $N-1$.

The classical effective methods for the Mordell Conjecture are for instance  described by J.P. Serre in his book \cite{serre} Chapter 5 and they can be used for some collection of curves to give their rational points. 
The  Chabauty-Coleman   method \cite{Chab}  and  \cite{Coleman} provides a sharp bound for the number of rational points   on a curve   transverse in an abelian variety of dimension strictly larger than its rank (here the varieties are defined over $\Q$). We refer to McCallum and Ponen  \cite{poonen} and Stoll \cite{Stoll} and their references for overviews on the method and for examples of curves of  genus 2  where the Chabauty-Coleman method can be used,  in combination with some decent argument, for finding their rational points.
Another effective method is given by  Manin and Dem'janenko (\cite{Demj},\cite{Manin}); it applies to curves  defined over a number field $k$ that admit $m$ different $k$-independent morphisms towards an abelian variety $A$ defined over $k$ with rank of $A(k) <m$. However, the method does not give an explicit  dependence of the  height of the $k$-rational points neither in terms of the curve nor in terms of the morphisms. Thus, to apply the method   such a dependence must be elaborated  example by example. Some  examples are given by Kulesz \cite{KuleszApplicazioneMD} and  Kulesz, Matera and Schost \cite{Kul3} for some  curves  with Jacobian isogenous to  a product of  special elliptic curves of rank one if the genus of the curve is $2$ and of rank two is the genus is $3$. \\

 New  insides to this topic has been brought by a different approach introduced by Bombieri, Masser and Zannier \cite{BMZ} in the framework of anomalous intersections. They  showed that  the set of  points of rank at most $N-2$ on a curve transverse in the torus $(\overline{\Q}^*)^N$ is finite. The rank of a point in an abelian variety is defined as  follows and an analogous definition can be given in semi-abelian varieties.
\begin{D}
\label{rank}
 The rank  of a point in an abelian variety $A$ is the minimal dimension  of a torsion variety containing the point.
\end{D}

The several implications and connections between  conjectures on anomalous intersections and  classical conjectures can be read for instance in the book of  Zannier \cite{Zan12} and in the survey article of the author \cite{via15}.\\

A fundamental aspect of the proof  of Bombieri, Masser and Zannier for curves in tori is  that it is effective, in the sense that it gives a bound for the height of the points of rank at most $N-2$. This however does not give any information on the rational points of the curve, as the rank of $\Q^N$ is not finite.  In \cite{ioannali}, the author extended their method to transverse curves in  a product of $N$ elliptic curves with CM. An interesting feature of her proof is that  it gives a method to bound the height of the points of rank $\le N-2$.  However, like for the effective methods mentioned above, it is far from being explicit and it is completely not uniform, as the dependence on $\Ci$ and $E$ is not specified and it basis on special intersection numbers and height functions. 

In \cite{CV} Corollary 1.6, S. Checcoli and the author used a Lehmer type bound to prove a  uniform bound for the height of the points of rank $<N/2$  on weak-transverse curves in $E^N$ where $E$ has CM. Furthermore, in \cite{via15} Theorem 6.4  the author proved a similar bound for points of rank $\leq N-1$  on transverse curves in $E^N$ where $E$ has CM. The bounds are explicit in the dependence on $\Ci$ but not effective in the dependence on $E^N$. The method cannot be extended to the non CM case, as a Lehmer type bound is not known in this case.  Moreover, even in the CM case, the Lehmer type bound is not proved in an effective way, making these bounds not effective in the  dependence on $E$.

 In the last years, we have been working to approach the problem with   explicit methods  aiming to  prove  new cases of the explicit Mordell Conjecture and to eventually find all the rational points on some curves. In \cite{ExpTAC} and \cite{EsMordell}   joint with S. Checcoli and F. Veneziano, we give  an explicit  bound for the N\'eron-Tate height of the set of points of rank one on curves of genus at least two in $E^N$ where $E$ is   without CM. The  non CM assumption is  technical and we handled there the easier case where the endomorphism ring of $E$ is  $\Z$.  
 
 %, thus simpler  to handle. If the elliptic curve has CM then the endomorphism ring is an order in the ring of integers $\mathcal{O}_K$ of {\bf an imaginary} quadratic field $K$. 

 %with $E$ {\bf an elliptic curve with CM, as in \cite{TAI}NON SONO ESPLICIT QUI, or without CM, see \cite{ExpTAC} and \cite{EsMordell}.} The CM assumption {\bf on $E$} is essential {\bf in \cite{TAI},} because the proof uses a Lehmer type boud not known in the non CM case.  In other cases, like \cite{ExpTAC}  and \cite{EsMordell}, the non-CM assumption is technical; in this case the endomorphism ring of $E$ is  $\Z$, thus simpler  to handle. If the elliptic curve has CM then the endomorphism ring is an order in the ring of integers $\mathcal{O}_K$ of {\bf an imaginary} quadratic field $K$. 

In this article we generalize the result of \cite{ExpTAC} in two directions. We first present the explicit computations  needed  to extend the method introduced in \cite{ExpTAC} to the case of $E$ with CM. This is done in Proposition \ref{main}.  
 To prove this Proposition, we   give some estimates on the degrees of morphisms on $E^N$ where $E$ has CM.  Then, using the adelic geometry of numbers, we  extend the  approximations done in \cite{ExpTAC}  to  $K$-lattices, for $K$ an imaginary quadratic field.

Secondly, we   give an explicit  bound for the N\'eron-Tate height  of points of rank  larger than one. More precisely  of points of rank $<N$ if $\Ci$ is a transverse curve in $E^N$ and  of points of  rank $< \max(r_\Ci-t_\Ci, t_\Ci) $ if $\Ci$ has genus at least $2$  and it is contained in a translate of minimal dimension $t_\Ci$ and in a torsion variety of minimal dimension $r_\Ci$.  In particular,  this applies  to points of  rank $< \max(N-t_\Ci, t_\Ci) $ if $\Ci$ is weak-transverse  in $E^N$ and contained in a translate of minimal dimension $t_\Ci$.

For transverse curves the result is optimal with respect to the rank, while for weak-transverse curves it is not optimal. We recall that an effective result with the optimal  rank $N-2$ in the weak-transverse case is an extremely deep statement which would imply the effective Mordell-Lang Conjecture for curves. This is  well known to be one of the most challenging open problem of this last century.  \\

% The special case of points of rank one in $E^N$ where $E$ is without CM was also treated in \cite{EsMordell} Theorem 1.2, here the bounds are  improved for rank one and generalised to larger rank and also to the CM case. 

We now introduce all  the notation needed to state our main theorem. We first recall the  definition of the two invariants just introduced above.   \begin{D}
  \label{Drank}
 For a curve $\Ci \subset  A$, we denote by $t_\Ci$ the minimal dimension of a  translate containing $\Ci$ and by $r_\Ci$ the minimal dimension of a torsion variety containing $\Ci$.
  \end{D}
    Notice that for a transverse curve $t_\Ci=r_\Ci=\dim A$ and for a weak-transverse curve $r_\Ci=\dim A$.   
    Let $E$ be an elliptic curve. We fix a Weierstrass equation
\begin{equation}\label{uno1}y^2=x^3+Ax+B,\end{equation} with $A$ and $B$ algebraic integers. Let $\Delta$ and $j$ be the discriminant and $j$-invariant of $E$. 
 Let  $\rend$ be the ring of endomorphisms of $E$, $K$ its field of fractions  and  $D_K$ its discriminant.
We embed $E^N$ into $\P_2^N$ via the affine equation \eqref{uno1} and then via the Segre embedding in $\P_{3^N-1}$.  
   Let $\hat{h}$ be the N\'eron-Tate height  on $E^N$ and  $h_W$ the absolute logarithmic Weil height in the projective space.
   The degree of a  curve   $\Ci\subset  E^N$ is the degree of its image in $\P_{3^N-1}$;  the normalised  height $h_2(\Ci)$  is defined in terms of the Chow form of the ideal of $\Ci$ as done  in  \cite{patrice} and $h(\Ci)$ is the canonical height of $\Ci$, as defined in \cite{patriceI}  (see Section \ref{SezioneAltezze} for more details). 
   
  We can then state our main theorem. 
 \begin{thm}\label{teorema}
 Let $\mathcal{C}\subset  E^N$ be a  curve of genus at least $2$  and $K=\mathrm{Frac}(\rend)$ with discriminant $D_K$.
 Then the set of points   $P\in \Ci$ of rank $r <\max(r_\Ci-t_\Ci, t_\Ci)$  is a set of N\'eron-Tate height explicitly bounded as \begin{align*}
\hat{h}(P)\leq& (N-r) |D_K|^{\frac{2t_\Ci+t_\Ci r+4r}{2(t_\Ci-r)}}\left(c_1(t_\Ci,r)  h_2(\Ci)(\deg \Ci)^{\frac{r}{t_\Ci-r}}+c_2(t_\Ci,r,E)(\deg \Ci)^{{\frac{t_\Ci}{t_\Ci-r}}}\right)+\\
&+(N-r)t_\Ci^2 C(E)+\frac{h(\Ci)}{\deg \Ci}.
\end{align*}
If furthermore  $\Ci$ is transverse we obtain
\begin{align*}
 \hat{h}(P)\leq& |D_K|^{\frac{2N+Nr+4r}{2(N-r)}}\left(c_1(N,r)  h_2(\Ci)(\deg \Ci)^{\frac{r}{N-r}}+c_2(N,r,E) (\deg \Ci)^{{\frac{N}{N-r}}}\right)+N^2 C(E).
\end{align*}
For any  integer $n>r$ we define  \begin{align*}
&c_1(n,r)=\left(2^{8n^2} 3^{2n^2} n^{9n^2}\right)^{\frac{1}{n-r}}\\
& c_2(n,r,E)=c_1(n,r)\left(3^{n}+n^2 C(E)\right),
\end{align*}
and 
%$H_i$ is the $i$-th harmonic number and 
$$ C(E)=\frac{h_W(\Delta)+3h_W(j)}{4}+\frac{h_W(A)+h_W(B)}{2}+4.$$
\end{thm}
As remarked in several works, the proof of Theorem \ref{teorema}  for points of a certain rank works similarly also for curves in a product of elliptic curves $E_1\times \cdots E_N$ instead of $E^N$. It this case it simply happens that some of the coefficients of the linear forms defining an algebraic subgroup  are zero. So Theorem \ref{teorema}  holds in $E_1\times \cdots E_N$ with $E_i$ where we fix equations of the type (\ref{uno1}) for every $E_i$. Using the universal property of the Jacobian one can easily state our theorems for curves with a factor of the Jacobian isogenous to $E_1\times \cdots E_N$.

The proof of a slightly sharper version of Theorem \ref{teorema} is given is Section \ref{DIMO}. We divide the theorem in two parts. The first one, given in Theorem \ref{casotras}, treats the case of a transverse curve.  This is proven with   typical tools of   diophantine approximation, based on some estimates for degree and heights done in  Section \ref{proofmain}  and  some refined argument of geometry of numbers presented in  Section \ref{atrans}. The second part is given in Theorem \ref{thmgen}. We reduce, via a geometric induction argument, the general case to  the case of a transverse curve. \\
%The  transverse case is proven with a typical proof of   diophantine approximation.  If $P$ is a point of rank $r$ then, by  the definition of rank,  there exists a torsion variety containing $P$  of dimension $r$. We first construct a translate $H+P$ of height and degree bounded in terms of $\Ci$ and $P$. We then use the Zhang inequality and the Arithmetic B\'ezout theorem on the intersection $\Ci\cap (H+P)$. The transversality of $\Ci$ ensures that $P$ is a component of such an intersection. Some technical manipulation  and a good choice of the parameters show that $\hat h(P)$ must be explicitly bounded.
%The extension to  curves of genus at least $2$ is delicate, indeed we cannot anymore ensure that $P$ is a component of any $\Ci\cap (H+P)$, because $\Ci$ might be contained in it.  To avoid this,
%  the rank of the points has to be smaller.

Theorem \ref{teorema} can be easily used  to describe the rational points of curves in $E^N$  under some conditions on the  rank of  the group $E(\Q)$ as $\rend$-module. We prove the following cases of the explicit Mordell Conjecture.  
 \begin{cor}
 \label{cor1}
 Let $E$ be an elliptic curve  and  let $\Ci$ be a curve of genus at least $2$ in $E^N$, both defined over a number field $k$. 

 Assume that 
 $E(k)$ is an $\rend$-module of  rank $r <\max(r_\Ci-t_\Ci,t_\Ci)$. Then the set of $k$-rational points  $\Ci(k)$ has  N\'eron-Tate height bounded as
\begin{align*}
\hat{h}(P)\leq& (N-r) |D_K|^{\frac{2t_\Ci+t_\Ci r+4r}{2(t_\Ci-r)}}\left(c_1(t_\Ci,r)  h_2(\Ci)(\deg \Ci)^{\frac{r}{t_\Ci-r}}+c_2(t_\Ci,r,E)(\deg \Ci)^{{\frac{t_\Ci}{t_\Ci-r}}}\right)+\\
&+(N-r)t_\Ci^2 C(E)+\frac{h(\Ci)}{\deg \Ci},
\end{align*}
where the constants are the same as in Theorem \ref{teorema}. 
 \end{cor}
Clearly if $E(k)$ has rank as $\rend$-module at most $r$ then all points in $\Ci(k)$ have rank at most $r$ in the sense of Definition \ref{rank}. So our main theorem directly implies Corollary \ref{cor1}.  This is the most general result we have obtained in the context of the  explicit Mordell Conjecture, and it is not covered by any  of our previous results, specifically we had explicit results only for points of rank one in the non CM case.

  Note that, by the Mordell-Weil Theorem, $E(k)$ is a finitely generated  abelian group  whose torsion-free part has $r$ generators if $\rend=\Z$ and $2r$ generators if $\rend=\Z+\tau \Z$. Thus to apply our theorem for transverse curves   it is sufficient to assume that $E^N$ has $k$-rank (in the sense of number of generators of $E^N(k)$)  smaller than $N^2$ in the non CM case and smaller than $2N^2$ in the CM case.  In particular this  extends (in the CM case) and makes explicit  the Manin and Dem'janenko method for  curves transverse in $E^N$. The coordinates morphisms restricted to the curve give the independent morphisms toward the elliptic curve  $E$ required by the Manin-Dem'janenko method. For $E$ without CM our and the Manin-Dem'janenko  assumption on the rank are the same; in the CM case our assumption is ${\rm rank}(E(k) <2N$ while the Manin-Dem'janenko assumption is ${\rm rank}(E(k)) <N$. 
  
  We finally remark that   the Chabauty-Coleman method in our setting of product of elliptic curves becomes trivial, indeed their assumption is that the rank of $E^N(\Q)$ is smaller than $N-1$. This  implies that $E$ has rank zero, and in a product of elliptic curves that at least one factor has rank zero. Then a simple use of the Arithmetic B\'ezout Theorem immediately gives a bound for the height of the rational points on the curve.\\
  
   In \cite{EsMordell} we gave a criterium for constructing transverse curves in $E^2$ where 
\begin{equation*}
\begin{split}
 y_1^2&=x_1^3+Ax_1+B, \\
y_2^2&=x_2^3+Ax_2+B
\end{split}
\end{equation*}
 are the equations of $E^2$ in $\P_2^2$ for affine coordinates $(x_1,y_1)\times (x_2,y_2)$.
 It is sufficient to cut  a curve in $E^2$ with  any polynomial of the  form $p(x_1)=y_2$ where  $p(x)\in \overline{\qe}[x]$.  In \cite{EsMordell} our examples  were given under the assumption that $E$ is non CM.  We can now extend them to any $E$, obtaining new iteresting examples.  In Section \ref{esempiE2} we prove the following 
\begin{cor}\label{caso_poly}  Assume that $E$ is defined over a number field $k$ and that $E(k)$ is an $\rend$-module of rank 1. 
 Let $\Ci$ be the projective closure  in $E^2$ of the affine curve cut by the additional equation $$p(x_1)=y_2,$$ with $p(x)=p_n x^n+\ldots+p_1x+p_0\in k[x]$ a non-constant polynomial  of degree $n$ with $m$ non-zero coefficients. 

Then every point  $P\in \Ci(k)$ has N\'eron-Tate height bounded as 
\[\hat h(P)\leq 2\cdot 10^{14} |D_K|^5  \left(h_W(p)+\log m+4C(E)\right)(2n+3)^2+4 C(E)\]
% CASO NON CM con i bound sharp DAVA \[\hat h(P)\leq 839810.268\left(h_W(p)+\log m+2\cotto(E)+3.01+2\ccinque(E)\right)(2n+3)^2+4\cquattro(E)\]
where $h_W(p)=h_W(1:p_0:\ldots:p_n)$ is the height of the polynomial $p(x)$ and the constant $C(E)$ is defined in Theorem \ref{teorema}.
\end{cor}

We finally remark that in \cite{ioant} the author  proved that the set of points of rank $N-2$ on a weak-transverse curve $\Ci$ in $E^N$ has bounded height; the proof is not effective and it uses in a non effective way the Vojta inequality, like it is used in the proof of the Mordell-Lang  Conjecture. However, in \cite{ioannali} the finiteness of the points of $\Ci$ of rank $\le N-2$ is effective and the bound on  their number depends  on $\Ci$, $E$, $N$ and on  the bound for their height in an effective way. %Moreover, by the Manin-Mumford conjecture the number  of torsion points on a curve of genus at least $2$  is effectively bounds. 
In Section \ref{dimocorcard}, we use our  Theorem \ref{teorema},  \cite{ioant} Theorem 1.3    and a projection argument to prove:
\begin{cor}\label{cardinalita}
For any  curve $\Ci$  of genus at least two embedded in $E^N$,  the set of points of $\Ci$ of rank $<\max(r_\Ci-t_\Ci,t_\Ci)$ has cardinality effectively bounded.
\end{cor}
Also here we can replace  $E^N$ by $E_1\times \dots \times E_N$, as explained above.
\section{Preliminaries}\label{SezioneAltezze}

 In this section we introduce the  notations and we recall several explicit relations between different height functions. 
  We also recall  some basic results in Arithmetic Geometry that play an important role in our proofs: the adelic Minkowski Theorem,  the Arithmetic B\'ezout  Theorem and the Zhang  Inequality. \\

Let $E$  be an elliptic curve defined over a number field $k$  by a fixed Weierstrass equation
\begin{equation}\label{Weq}
E: y^2=x^3+Ax+B
\end{equation}
with $A$ and $B$  in the ring of  integers  of $k$ (this assumption is not restrictive).  We denote the  discriminant of $E$ by
\[
 \Delta=-16(4A^3+27B^2)\] and the $j$-invariant  by \[j=\frac{-1728(4A)^3}{\Delta}.
\]
 We consider  $E^N$ embedded in $\P_{3^N-1}$ via the following composition map 
\begin{equation}\label{embe}
  E^N \hookrightarrow \P_2^N \hookrightarrow \P_{3^N-1}
\end{equation}
where the first map sends a point $(X_1,\ldots,X_N)$ to $((x_1,y_1),\dotsc,(x_N,y_N))$ (the $(x_i,y_i)$ being the affine coordinates of $X_i$ in the Weierstrass form of $E$) and the second map is the Segre embedding. 
Degrees and  heights are computed with respect to this fixed embedding.

\subsection{Heights of points}

If $P=(P_1:\ldots:P_n)\in \mathbb{P}_n(\overline{\Q})$ is a point in the projective space, then  the absolute logarithmic Weil height of $P$ is defined as
\[h_W(P)=\sum_{v\in \mathcal{M}_K} \frac{[K_v:\Q_v]}{[K:\Q]}\log \max_i \{\abs{P_i}_v\}\]
where $K$ is a field of definition for $P$ and $\mathcal{M}_K$ is its set of places. If $\alpha\in\overline\Q$ then the Weil height of $\alpha$ is defined as $h_W(\alpha)=h_W(1:\alpha)$.

We also define another height which differs  from the Weil height  at the Archimedean places:
\begin{equation}\label{Defh2}
h_2(P)=\sum_{v\text{ finite}}\frac{[K_v:\Q_v]}{[K:\Q]}\log \max_i \{\abs{P_i}_v\} +\sum_{v\text{ infinite}}\frac{[K_v:\Q_v]}{[K:\Q]}\log \left(\sum_i \abs{P_i}_v^2\right)^{1/2}.\end{equation}
%If $x$ is an algebraic number, we denote by $h_{\infty}(x)$ the contribution to the Weil height coming from the archimedean places, more precisely:
  %\[
% h_\infty(x)=\sum_{v\text{ infinite}}\frac{[K_v:\Q_v]}{[K:\Q]} \max \{\log \abs{x}_v, 0\}.\]

 For a point $P\in E$ we denote by $\hat h(P)$ its N\'eron-Tate height as defined in \cite{patriceI}  {\bf (which} is one third of the usual N\'eron-Tate height used also in   \cite{ExpTAC}). 
%(NOTA: perche dare referenza a patrice per il NT height?----perche' usiamo la sua normalizzazione che e' un terzo di quella di Silverman mi pare)

If $P=(P_1,\dotsc,P_N)\in E^N$,  
 then for $h$ equal to $h_{W},h_2$ and $\hat{h}$ we define
\begin{equation*}
   h(P)=\sum_{i=1}^N h(P_i).
\end{equation*}
%If $E$ is an elliptic curve given by a Weierstrass equation as in \eqref{Weq}, we define  \begin{equation}\label{hWeierE} h_{\mathcal{W}}(E)=h_W(1:A^{1/2}:B^{1/3})
%\end{equation} to be the absolute logarithmic Weil height of the projective point $(1:A^{1/2}:B^{1/3})$. 

The following proposition directly  follows from \cite[Theorem 1.1]{SilvermanDifferenceHeights} and \cite[Proposition 3.2]{EsMordell}.
\begin{propo}\label{confrontoaltezze}
 
%For $P\in \mathbb{P}_m$, 
%\begin{equation}\label{hh2}
%h_W(P)\leq h_{2}(P)\leq h_W(P)+\log(m+1)/2.
%\end{equation}

%For $P\in E$,  
%\begin{equation}\label{hWhath}-\frac{(h_W(j)+2h_W(\Delta)+2 h_\infty(j))}{24}-0.973 \leq \frac{\hat h(P)}{3}-\frac{h_W(x(P))}{2}\leq \frac{h_W(\Delta)+h_\infty(j)}{12}+1.07.
%\end{equation}

For $P\in E^N$, 
 \begin{equation*}
 -N C(E)\leq h_2(P)-\hat h (P)\leq N C(E),
\end{equation*}
where 
$$C(E)=\frac{h_W(\Delta)+3h_W(j)}{4}+\frac{h_W(A)+h_W(B)}{2}+4.$$
%$$ \ccinque(E)=\frac{h_W(\Delta)+h_\infty(j)}{4}+\frac{h_W(j)}{8}+\frac{h_W(A)+h_W(B)}{2}+3.724,$$
%$$ \cquattro(E)=\frac{h_W(\Delta)+h_\infty(j)}{4}+\frac{h_W(A)+h_W(B)}{2}+4.015.$$        

%NON SI USAVA
%Moreover, if $P\in E(\qe)$ then
 %\begin{equation}\label{zimmer}
%-\frac{(3h_{\mathcal W}(E)+7\log2)}{2}\leq  h_{W}(P)-\hat h(P)\leq 3h_{\mathcal W}(E)+6\log 2.
%\end{equation}
%and we can replace  the constants in (\ref{h2hath}) by
%$$\ccinque(E)=\min\left(\frac{\log\abs{\Delta}+h_\infty(j)}{4}+\frac{h_W(j)}{8}+\frac{\log(\abs{A}+\abs{B}+3)}{2}+2.919,3h_{\mathcal W}(E)+4.709\right),$$
%$$ \cquattro(E)=
   %             \min\left(\frac{\log\abs{\Delta}+h_\infty(j)}{4}+\frac{\log(\abs{A}+\abs{B}+3)}{2}+ 3.21, \frac{3h_{\mathcal W}(E)}{2}+2.427\right).$$

\end{propo}
Further details on the relations between the different height functions defined above can be found in \cite{EsMordell}, Section 3.
\medskip

\subsection{Heights of varieties} 
For  a subvariety $V\subset \P_m$  we denote by $h_2(V)$ the normalised height of $V$ {defined}  in terms of the Chow form of the ideal of $V$, as done  in  \cite{patrice}. This height extends the height $h_2$ defined for points by formula \eqref{Defh2}  (see \cite{BGSGreen}, equation (3.1.6)). 
We also consider  the canonical height $h(V)$, as defined in \cite{patriceI}; when the variety $V$ reduces to a point $P$, then $h(P)=\hat h(P)$  (see \cite{patriceI}, Proposition 9).

\subsection{The degree of varieties}

The degree of  an irreducible variety $V\subset  \P_m$ is the maximal cardinality of a finite intersection $V\cap L$, with $L$ a linear subspace of dimension equal to the codimension of $V$. The degree is often conveniently computed as an intersection product.

If $X(E,N)$ is the image of $E^N$ in $\mathbb{P}_{3^N-1}$ via the above map, then by \cite{ExpTAC}, Lemma 2.1 we have \begin{equation}\label{degSegre}\deg X(E,N)=3^N N!.\end{equation}

\subsection{ The Arithmetic B\'ezout Theorem}
 The following explicit result is proven by Philippon in  \cite{patrice}, Th\'eor\`eme 3. It describes  the behavior of the height for intersections. \begin{thm}[Arithmetic B\'ezout theorem]\label{AriBez}
 Let $X$ and $Y$ be irreducible closed subvarieties of $\P_m$ defined over  the algebraic numbers. If $Z_1,\dotsc,Z_g$ are the irreducible components of $X\cap Y$, then
 \[
  \sum_{i=1}^g h_2(Z_i)\leq\deg(X)h_2(Y)+\deg(Y)h_2(X)+\csei(\dim X,\dim Y, m)\deg(X)\deg(Y)
 \]
where
\begin{equation}\label{costaBez}
 \csei(d_1,d_2,m)=\left(\sum_{i=0}^{d_1}\sum_{j=0}^{d_2} \frac{1}{2(i+j+1)}\right)+\left(m-\frac{d_1+d_2}{2}\right)\log2.
\end{equation}

\end{thm}

\subsection{The Zhang Inequality}
 In order to state  Zhang's inequality,  we define the essential minimum $ \mu_2(X)$ of an irreducible algebraic subvariety $X\subset \P_m$ as 

\[
  \mu_2(X)=\inf\{\theta\in\R\mid\{P\in X\mid  h_2(P)\leq\theta\}\text{ is Zariski dense in }X\}.
\] 

The following result is due to Zhang \cite{Zhang95}, Theorem 5.2: \begin{thm}[Zhang inequality]\label{Zhang}
Let $X\subset \P_m$ be an irreducible algebraic subvariety. Then 
\begin{equation}\label{zhangh2}
\mu_2(X)\leq\frac{h_2(X)}{\deg X}\leq(1+\dim X)\mu_2(X).
\end{equation}
\end{thm}
We also define a different essential minimum for subvarieties of $E^N$, relative to the height function $\hat h$:
\[
 \hat\mu(X)=\inf\{\theta\in\R\mid\{P\in X\mid \hat h(P)\leq\theta\}\text{ is Zariski dense in }X\}.
\]
Using the definitions and a simple limit argument, one sees that Zhang's inequality holds also with $\hat{\mu}$, namely 
\begin{equation}\label{Zhangh^}\hat\mu(X)\leq\frac{h(X)}{\deg X}\leq(1+\dim X)\hat\mu(X).\end{equation}

If $X$ is an irreducible subvariety in $E^N$, using Proposition \ref{confrontoaltezze} we  have
\begin{equation}\label{mu2mu^}
-N\cquattro(E)\leq \mu_2(X)-\hat\mu(X)\leq  N\ccinque(E)
\end{equation}
where the  constant $C(E)$ is defined in Proposition \ref{confrontoaltezze}.

%Finally, using \eqref{mu2mu^}, \eqref{zhangh2} and \eqref{Zhangh^} we  get: \begin{equation}\label{confrontohh2}
% \frac{h_2(X)}{1+\dim X}-N\ccinque(E)\deg X\leq h(X)\leq (1+\dim X)\left(h_2(X)+N\cquattro(E)\deg X\right).
%\end{equation}

\subsection{Complex Multiplication}\label{CM}
 We denote by $\rend$ the ring of endomorphisms of $E$. We recall that an elliptic curve $E$   is non  CM if $\mathrm{End}(E)$ is isomorphic to $\mathbb{Z}$, while $E$ is CM if $\mathrm{End}(E)$ is isomorphic to an order in  $K=\Q(\sqrt{D})$, for some squarefree negative integer $D$.   In particular,  $\mathrm{End}(E)$ is isomorphic to $\mathbb{Z}+\mathfrak{f}\mathcal{O}_K$ where $\mathcal{O}_K$ is the ring of integers of $K$ and $\mathfrak f^2$ divides $D_K$, the discriminant of $K$.

 If $D$ is congruent to 1 modulo 4, then  $D_K=D$  and $\mathcal{O}_K=\Z[(1+\sqrt D)/2]$, so $\mathfrak f=1$ and $\rend=\mathbb{Z}[(1+\sqrt D)/2]$. If $D$ is not congruent to 1 modulo 4, then $D_K=4D$ and $\mathfrak f\in\{1,2\}$: if $\mathfrak f=1$ then $\rend=\Z[\sqrt D]$, while if $\mathfrak f=2$ then $\rend=\Z[2\sqrt D]$.
In all cases $\rend=\Z[\tau]$ for the above specified $\tau$ and we say that $E$ has CM by $\tau$.

\subsection{The adelic Minkowski Theorem} 

In this section we follow the notations and  arguments of \cite{BG06} Appendix 6. The adelic Minkowski Theorem has been  proved by Bombieri and Vaaler \cite{BomVal}  Theorem 3, where they use it to  prove Siegel's Lemma over number fields.

Let $K$ be an imaginary quadratic number field with ring of integers $\mathcal{O}_K$ and let $r$ and $N$ be positive integers. A $K$-lattice  $\Lambda$ of rank $r$  is an $\mathcal{O}_K$-module of rank $r$ such that $\Lambda \otimes _{\mathcal{O}_K}K$ is a $K$-vector space of dimension $r$ (see \cite{BG06}, Definition C.2.5)

 Consider a matrix  $M\in\mathrm{Mat}_{r\times N}(\rend)$ of rank $r$. Then  the rows of $M$ generate an $\mathcal{O}_K$-module and so a $K$-lattice $\Lambda\subset  \rend^N\subset  K^N$. We define the determinant of $\Lambda$ as $\det\Lambda=\sqrt{\det(M\overline{M}^t)}$, where $\overline{M}^t$ is the transpose of the complex conjugate  of $M$. This is also the covolume of a fundamental domain of $\Lambda$. 
%{\bf For a place $v$ of $k$ we denote by $k_v$ the completion of $k$ at $v$ and by $\Lambda_v$ the closure of $\Lambda$ in $k_v$.

%We denote by $k_\mathbb{A}$ the ad\`ele ring of $k$ (see definition C.1.4 of \cite{BG06}).}

The following is a special case of the adelic  Minkowski second theorem proved in \cite{BomVal}  Theorem 3 and also in \cite{BG06} Theorem C.2.11 and  Section C.2.18. 
%(note that  $\mathrm{vol}(\Lambda_{\infty})\le (\mathrm{vol}\mathcal{O}_K)^r \mathrm{vol}\Lambda^2\le \frac{|D_{K}|^{r/2}}{2^r}\mathrm{vol}\Lambda^2$). 
\begin{thm}\label{Minko} Let $K$ be an imaginary quadratic field with discriminant $D_K$ and $\Lambda$ a $K$-lattice of rank $r$ as defined above. Then there exist generators $u_1,\dots u_r$ of $\Lambda$ with  euclidean norm $||u_i||=\lambda_i$ such that  
\begin{equation*}
 \omega_{2r} (\lambda_1\dotsm\lambda_r)^2 \leq 2^{r}  |D_{K}|^{r/2} (\det \Lambda)^2,
\end{equation*}
where
$
 \omega_{2r}={\pi^{r}}/{r!}
$
is the volume of the unit ball in $\R^{2r}$.
\end{thm}
\begin{proof}
To deduce this version from  \cite{BomVal}  Theorem 3 or  \cite{BG06} Theorem C.2.11, we 
note that $\Lambda\otimes_{\mathcal{O}_K}K$ makes $\Lambda$ a full $K$-lattice as defined by Bombieri and Vaaler. In addition $\Lambda$ is embedded in $\rend^N$ and so in $K^N$. Then it remains to estimate ${\rm Vol}(S)$. This is done 
 in \cite{BG06} C.2.18.
We choose the standard $\Z$-basis  of $\mathcal{O}_K$ and then  identify $K^r$ with $\qe^{2r}$ and in turn see $\Lambda$ embedded in  $K^N$ and $\qe^{2N}$. Then our $K$-lattice $\Lambda$ may be viewed as an $\mathbb{R}$-lattice $\Lambda_\infty$ of rank $2r$ embedded in  $\mathbb{R}^{2N}$. 

The classical Minkowski Theorem tells us that  there exist generators $v_i$ for $i=1,\dots,2r$ of $\Lambda_\infty$ (defining the classical successive minima) such that

$$ \omega_{2r} ||v_1||\dots ||v_{2r}|| \leq 4^{r}  (\mathrm{vol} \Lambda_\infty).$$

A computation shows that $\mathrm{vol}(\Lambda_{\infty})\le (\mathrm{vol}\mathcal{O}_K)^r \mathrm{vol}\Lambda^2$, where $\mathcal{O}_K$ is considered as a $\qe$-lattice of rank $2$.  In addition $\mathrm{vol}\mathcal{O}_K=\frac{|D_{K}|^{1/2}}{2}$   by classical results and $ \mathrm{vol}\Lambda=\det \Lambda$ by definition.
 Finally one can extract  from the $v_i$ an $\mathcal{O}_K$-basis $u_1,\dots u_r$ of $\Lambda$ such that $||u_i||\le ||v_{2i-1}||$. So $(\lambda_1\dotsm\lambda_r)^2\leq ||v_1||\dots ||v_{2r}||$ that together with the above bounds gives the theorem.
\end{proof}

Note that, from the proof, one sees that   $||u_i||=\lambda_i$ are the successive minima of $\Lambda$, where the  definition is as follows. Let  $S$ be the unitary complex ball of $\mathbb{C}^{r}$ with respect to the Lesbegue measure.  The $i$-th successive minimum  is 
\[
 \lambda_i=\inf\{\lambda \geq 0 \mid \lambda\cdot S \text{ contains }i\text{ linearly independent vectors of }\Lambda  \}.
\]
This definition is equivalent to the one given in \cite{BG06} C.2.9 as explained in \cite{BomVal} (3.2).

\subsection{Algebraic Subgroups}\label{prelim-S}
In this subsection we recall the several different descriptions of the  algebraic subgroups of $E^N$.

By  the uniformization theorem there exists a unique lattice $\Lambda_0 \subset  \C$ such that the map $\C \rightarrow \PP^2(\C )$ given by
\[  z \longmapsto \left\{ \begin{array}{lll}
[\wp_{\Lambda_0}(z) : \wp'_{\Lambda_0} (z) : 1 ] & , & z \not\in {\Lambda_0} \\{}
[0:1:0] &, & z \in {\Lambda_0}
\end{array}    \right. \]
induces an isomorphism $\C  / {\Lambda_0} \stackrel{\sim}{\rightarrow} E(\C ) \subset  \PP^2(\C )$ of complex Lie groups (see~\cite[Theorem VI.5.1]{SilvermanArithmeticEllipticCurves}. Here $\wp_{\Lambda_0}$ denotes the Weierstrass $\wp$-function associated to the lattice ${\Lambda_0}$. The $N$-th power of this isomorphism gives the analytic uniformization
\[ \wp_e: \C ^N / {\Lambda_0}^N \stackrel{\sim}{\rightarrow} E^N(\C ) \]
of $E^N$. Note that $\C ^N$ can be identified with the Lie algebra  of $E^N(\C )$ and the composition of the canonical projection $\pi: \C ^N \rightarrow \C ^N / {\Lambda_0}^N$ with $\wp_e$ can be identified with the exponential map of the Lie group $E^N(\C )$. By~\cite[8.9.8]{BG06}, the set of abelian subvarieties of $E^N$ is in natural bijection with the set of complex vector subspaces $W \subset  \C ^N$ for which $W \cap {\Lambda_0}^N$ is a lattice of full rank in $W$. Such a $W$ corresponds to the abelian subvariety $B$ of $E^N$ with $B(\C ) = (\wp_e \circ \pi)(W)$ and we  identify $W$ with the Lie algebra of $B$.

The \emph{orthogonal complement} of an abelian subvariety $B \subset  E^N$ with Lie algebra $W_B \subset  \C ^N$ is the abelian subvariety  $B^\perp$ with Lie algebra $W^\perp_B$ where $W^\perp_B$ denotes the orthogonal complement of $W_B$ with respect to the canonical Hermitian structure of $\C ^N$. Note that $W_B^\perp \cap {\Lambda_0}^N$ is a lattice of full rank in $W_B^\perp$ and its volume can be estimated using the Siegel Lemma over number fields of Bombieri and Vaaler \cite{BomVal}.

Let $B$ be an irreducible torsion variety of $E^N$ with $\codim\sotto=r$. By \cite{Masserwustholz} Lemma 1.3, $B$ is a component of an algebraic subgroup  $H$ of codimension $r$ and of degree $\le \deg B T_0$ where $T_0$ is  the number of components of $H$ and it is absolutely bounded. The  Lie algebra of the zero component of $H$ is the kernel of a linear map $\varphi_\sotto:\C^N \to \C^r$. We identify $\varphi_\sotto$ with the induced  morphism $\varphi_\sotto:E^N \to E^{r}$. Then  $\ker \varphi_\sotto=\sotto+\tau$ with $\tau$ a torsion set of  cardinality  $T_0$.
In turn $\varphi_\sotto$ is identified with a matrix in $\mathrm{Mat}_{r\times N}(\mathrm{End}(E))$ of rank $r$  and, using the geometry of numbers, we can choose the matrix   representing $ \varphi_\sotto$  in such a way that the degree of   $\sotto$ is essentially the product of  the squares of the norms  of the rows of the matrix.

More precisely, given $B\subset  E^N$ an algebraic subgroup of rank $r$, we associate it with a matrix in $\mathrm{Mat}_{r\times N}(\mathrm{End}(E))$ with rows  $u_1,\dotsc,u_r$ 
such that the euclidean norm $||u_i||$ of $u_i$ equals the $i$-th successive minimum of the $K$-lattice $\Lambda=\langle u_1,\dotsc, u_r\rangle_{\mathcal{O}_K}$  for $K=\mathrm{Frac}(\rend)$ and $\mathcal{O}_K$ its ring of integers. In Subsection \ref{sec-c6} we show that 
\begin{equation*}
 \deg B\leq 3^N N!\left(12^{N-1} 2\right)^r\prod_{i=1}^r||u_i||^2.
\end{equation*}

Combining this with Minkowski's Theorem recalled above we get
\begin{equation}\label{gradodet}
\frac{\deg B}{(\det \Lambda)^2}\leq 3^N N!\left(12^{N-1} 2\right)^r \frac{2^r|D_K|^{\frac{r}{2}}}{\omega_{2r}}.
\end{equation}

\section{Estimates for the degree of morphisms}\label{section3}

 In this section we give sharp bounds for the degree of an algebraic subgroup of $E^N$, where the embedding of $E^N$ in the projective space is fixed  as done in \eqref{embe}.  

An  algebraic subgroup $H$ of $E^N$ of codimension $s$ is defined by $s$ equations
\begin{align*}
L_1&(X_1,\dotsc,X_N)=0, \\
&\vdots\\ 
L_s&(X_1,\dotsc,X_N)=0
\end{align*}
{where $L_i(X_1,\ldots,X_N)=l_{i_1}X_1+\cdots+l_{i_N}X_N$ are  morphisms from $E^N$ to $E$; here  the  coefficients $l_{i_j}\in \mathrm{End}(E)$  are expressed by certain rational functions; similarly the $+$ that  appears in this expression is the addition map in $E^N$, which is expressed by a rational function of the coordinates.

More precisely, if the $X_i$'s are all points on $E$ with {affine} coordinates $(x_i,y_i)$ in $\P_2$, then $L_i(\mathbf{X})$ are also points on $E$ with coordinates in $\P_2$ $(x(L_i(\mathbf{X})), y(L_i(\mathbf{X})))$ which are rational functions of the  $x_i$'s and $y_i$'s.

The purpose of this section is to study the rational functions $x(L_i(\mathbf{X}))$ and  to bound the sums of {\bf their} partial degrees.

\subsection{Estimates for degrees of rational functions}\label{secpolinomi}
 Recall that if $\frac{f_i}{g_i}$ are rational functions, with $f_i,g_i$ polynomials in several variables, we denote by $d(f_i/g_i)$ the maximum of the sums of the partial degrees of both $f_i,g_i$. Then
\begin{align}
\label{stimaprodottofunzioni}\frac{f}{g}&=\prod_{i=1}^r\frac{f_i}{g_i}  & d(f/g)&\leq \sum_{i=1}^r d(f_i/g_i)\\
\label{stimasommafunzioni}\frac{f}{g}&=\sum_{i=1}^r\frac{f_i}{g_i} & d(f/g)&\leq \sum_{i=1}^r d(f_i/g_i) 
\end{align}
where $d(f/g)$ is the bound for the sum of partial degrees of the product  (in \eqref{stimaprodottofunzioni}) and of the sum (in \eqref{stimasommafunzioni}) of the $f_i/g_i$'s respectively.

\subsection{Estimates for endomorphisms }\label{secmoltiplicazionem}

In this section we are going to estimate the degree of the rationa functions on $E$ giving the mulitplication by an element of $\rend$. Recall that if $E$ has CM by $\tau$ and $\alpha\in \mathrm{End}(E)$, then $\alpha=m+n\tau$ for some integers $m,n$. We denote by $|\cdot |$ the complex absolute value.  Notice that since $\alpha$ is an algebraic integer then $|\alpha|\geq 1$ for every $\alpha\neq 0$.

 From the  theory of elliptic functions we can deduce the following result. This will be used, together with other estimates for the  non CM case,  to bound the degree of an algebraic subgroup.
%We have the following result (see [CITA Advanced Topics Arithm Ell. Curves di Silverman], Exercice 2.24).
\begin{lem}\label{boundmultalfa} 

Let $P=(x,y)$ be a point in $E$ and let $\alpha\in \mathrm{End(E)}$. Then
$$\alpha(x,y)=\left(f(x),\frac{y f'(x)}{\alpha}\right)$$
with $f$ a rational function with partial degrees bounded by $|\alpha|^2$. 
%{\bf (NOTA: Cambiato la $f$ all'inizio in $\mathfrak f$ che mi pare serva DOVE?)}

In particular the coordinates of $\alpha(x,y)$ are rational functions of $x$ and $y$ whose sums of the partial degrees are upper bounded by $2|\alpha|^2$ and lower bounded by $|\alpha|^2$.
\end{lem}
\begin{proof}
Let $\alpha\in\mathrm{End}(E)$. Then 
$$\alpha(x,y)=(f(x)+yg(x),h(x)+y l(x))$$
for $f,g,h,l$ rational functions. Since the multiplication by $\alpha$ is an isogeny on $E$ we have that $\alpha(-P)=-\alpha P$ for every $P\in E$, thus
\[\alpha(x,-y)=(f(x)-yg(x),h(x)-yl(x))=-\alpha(x,y)=(f(x)+yg(x),-h(x)-yl(x))\]
from which we deduce $g=h=0$.

%Now let $\omega=\frac{dx}{y}$ be the invariant differential.
%Since $\alpha^{*}(\omega)=\alpha\omega$ we have $l(x)=f'(x)/\alpha$. {\bf eviteri di parlare di $\omega$ basta dire che 
 Notice that if $(x,y)=(\wp (z), \wp'(z))$, with $\wp$ the Weiertrass $\wp$-function, then $\alpha(x,y)=(\wp (\alpha z),  \wp'(\alpha z))$.
So we obtain  $\alpha(x,y)=\left(f(x),\frac{y f'(x)}{\alpha}\right)$ where $f$ is a rational function. Writing $f(x)=\frac{a(x)}{b(x)}$ with $a,b$ polynomials, by Theorem 10.14 of \cite{Cox}, we have that $\deg a(x)=\deg b(x)+1=|\alpha|^2$.
 Therefore the  sums of the partial degrees of $f(x)$ and $yf'(x)$ are smaller than $2|\alpha|^2$
and bigger than $|\alpha|^2$.
\end{proof}

\subsection{Estimates for group morphisms}\label{eq_sub} 
We first estimate the degree of an algebraic subgroup of codimension one, thus defined by one linear equation. Then we generalize the estimates to algebraic subgroups of any codimension, thus defined by several linear equations.

Consider  first $M$ points on $E$, $P_1=(x_1,y_1), P_2=(x_2,y_2),\ldots, P_M=(x_M,y_M)$ and let $P_{M+1}=(x_{M+1},y_{M+1})=P_1\oplus P_2\oplus\ldots \oplus P_M$ be their sum. 
Then $x_{M+1}$ and $y_{M+1}$ are certain explicit rational functions  of the coordinates  $x_i,y_i$ for $i=1,\ldots,M$. 

 We proved in \cite{ExpTAC},  Section 6.3.1, that if the points $(x_i,y_i)$, for $i=1,\ldots,M$ have coordinates given by  some rational functions in certain variables,  whose sums of the partial degrees  are bounded by $d_i$, then the  sum $d$ of the partial degrees   of the functions $x_{M+1},y_{M+1}$ in the variablesf $x_i,y_i$ is bounded as
\begin{align}
\label{sommaNpuntigrado}d&\leq 12^{M-1} d_1 + \sum_{i=2}^{M} 12^{M-i+1}d_i\leq 12^{M-1} \sum_{i=1}^{M}d_i.
\end{align}

Let us consider the {morphism $L:E^N\rightarrow E$ defined as} 
\begin{equation*}
L(\mathbf{X})=l_1X_1+\dotsb+l_NX_N
\end{equation*}  
{where $l_i\in \mathrm{End}(E)$ and $X_i=(x_i,y_i)$ is in the $i$-th factor of $E^N$.} Then $L(\mathbf{X})$ is also a point on $E$ with coordinates $(x(L(\mathbf{X})), y(L(\mathbf{X})))$ that are rational functions in the coordinates $(x_i,y_i)$ of all $X_i$'s. We want to  bound the sum of the partial degrees of the rational function $x(L(\mathbf{X}))$.

Let us set $d(L)=d(x(L(\mathbf{X})))$ to be the sum of the partial degrees in the numerator and denominator of $x(L(\mathbf{X}))$.

Now combining inequality \eqref{sommaNpuntigrado} with the bounds from Subsection \ref{secmoltiplicazionem} we obtain 
\begin{equation}\label{combinazioneNpuntigrado}
d(L)\leq12^{N-1}2\left(\sum_{i=1}^N |l_i|^2\right).
\end{equation}
%{\bf IO TOGLIEREI} \begin{remark}
%The above results holds regardedless if $E$ has CM or not.
%If $E$ is non CM, then the $l_i$ are all integers and, using also the sharp bounds obtained in \cite[Lemma 9.1]{EsMordell} for the degree of the rational functions giving the multiplication by integers on $E$, we get the sharper estimate 
%\begin{equation}\label{gradononCM} d(L)\leq 12^{N-1}\frac{3}{2}\left(\sum_{i=1}^N |l_i|^2\right).\end{equation}
%\end{remark}

\section{Estimates for degree and height of a translate}\label{proofmain}
In this section we give explicit bounds for the degree of a translate $H+P$ in $E^N$ in terms of the  coefficients of the linear forms defining $H$ and we bound  its height in terms of the height of $P$. 

%More precisely, let $P\in E^N$ be a point and let $H$ be a component of the  algebraic subgroup defined by a certain $s\times N$ matrix with coefficients in  $\mathrm{End}(E)$. We want to bound the height and the degree of $H+P$ explicitelly in terms of $\hat h(P)$ and the rows of the associated matrix.

%We use the bounds for the height of Section \ref{SezioneAltezze} and for the degree of Section \ref{section3}.  
For a vector, we denote by $||\cdot||$ its euclidean norm.

\subsection{A bound for the degree}\label{sec-c6} 
To bound the degree of a translate, we do the same construction of \cite{ExpTAC},  recalled here for clarity, and we use the explicit bounds computed above for the CM case.
We first consider an  algebraic subgroup given by a single equation in $E^N$.
Then  we apply the Segre embedding and see  this subgroup as a subvariety of $\P_{3^N-1}$. In doing this we must be careful in selecting irreducible components. Finally we apply inductively B\'ezout's theorem for the case of several equations.

\medskip

Let $U$ be a matrix in $\mathrm{Mat}_{s\times N}(\rend)$ with rows $\mathbf{u}_1,\dotsc,\mathbf{u}_s\in \rend^N$ and let $H\subset  E^N$ be an irreducible component of the algebraic subgroup associated with the matrix $U$ (see Subsection \ref{prelim-S}).

If $X_1=(x_1,y_1),\ldots, X_N=(x_N,y_N)$ are points on $E$ and  $\mathbf{v}=(v_1,\ldots,v_N)\in \rend^N$ is a vector, we denote $\mathbf{v}(\mathbf{X})=v_1 X_1+\ldots+v_N X_N$.

As remarked in the previous section, $\mathbf{v}(\mathbf{X})=(x(\mathbf{\mathbf{v}}(\mathbf{X})), y(\mathbf{v}(\mathbf{X})))$ is a point in $E$ and $x(\mathbf{v}(\mathbf{X}))$ is a rational function of the $x_i,y_i$'s.

Let now $P=(P_1,\ldots,P_N)\in E^N$ be a point.
Take the $k$-th row $\mathbf{u}_k\in \rend^N$ of $U$  and consider the equation
\[x(\mathbf{u}_k(\mathbf{X}))=x(\mathbf{u}_k(P))\]
with $\mathbf{X}=(X_1,\ldots,X_N)\in E^N$ as before.

Clearing out the denominators the previous equation can be written as \[f_{\mathbf{u}_k,P}(x_1,y_1,\dotsc,x_N,y_N)=0,\]
where $f_{\mathbf{u}_k,P}$ is a polynomial of degree bounded by $d(\mathbf{u}_k)$ (see formula \eqref{sommaNpuntigrado})  which defines a variety in $\P_2^N$. Applying the Segre embedding, we want to study this variety as a subvariety of $\P_{3^N-1}$.

The Segre embedding induces a morphism between the fields of rational functions, whose effect on the polynomials in the variables $(x_1,y_1,\dotsc,x_N,y_N)$ is simply to replace any monomial in the variables of $\P_2^N$ with another monomial in the new variables, without changing the coefficients; the total degree in the new variables is the sum of the partial degrees in the old ones.

Recall that in Section \ref{SezioneAltezze} we defined $X(E,N)$ as the image of $E^N$ in $\P_{3^N-1}$.

\medskip

Denote by  $Y'_{k}\subset  \P_{3^N-1}$ the zero-set of the polynomial  $f_{\mathbf{u}_k,P}(x_1,y_1,\dotsc,x_N,y_N)$ after embedding $\P_2^N$ in $\P_{3^N-1}$. 

Now consider an irreducible component of the translate in $E^N$  defined by \[\mathbf{u}_k(\mathbf{X})=\mathbf{u}_k(P)\] and denote by $Y_k$ its image  in $\P_{3^N-1}$. We want to  bound the degree of the hypersurfaces $Y_k$.
Notice that 
\[Y_k\subset  Y'_k\cap X(E,N)\] and it is a component.
This is because setting the first coordinate of $\mathbf{u}_k(\mathbf{X})$ equal to $x(\mathbf{u}_k(P))$ defines two cosets, $\mathbf{u}_k(\mathbf{X})=\mathbf{u}_k(P)$ and $\mathbf{u}_k(\mathbf{X})=-\mathbf{u}_k(P)$.

By B\'ezout's theorem
\begin{align*}
 \deg Y_k\leq \deg X(E,N)\deg{ Y'_k}\leq 3^N N! 12^{N-1} 2||\mathbf{u}_k||^2,
\end{align*}
where the last inequality follows from formula \eqref{combinazioneNpuntigrado}.

\smallskip

In a similar way, considering all the rows we get
\begin{equation}\label{gradorighe}
\deg (H+P)\leq  \deg X(E,N)\deg Y'_1\dotsm\deg Y'_s\leq 3^N N! \left(12^{N-1} 2\right)^s\prod_{i=1}^s||\mathbf{u}_i||^2.
\end{equation}

\subsection{A bound for the height}\label{subsec-alttrasl}
%In this section we assume that $E$ has complex multiplication given by $\tau$ and we set $K=\Q(\tau)$.
 The aim of this section is to prove the following:
\begin{propo}\label{stimaaltezzanormalizzatatraslato}
Let $E$ be an elliptic curve with CM by $\tau$.  Set $K=\Q(\tau)$, let $\mathcal{O}_K$ be its ring of integers and $D_K$ its discriminant.
Let $P\in E$ be a point. Let $H\subset  E^N$ be a component of the  algebraic subgroup of codimension $s$ associated with an $s\times N$ matrix with rows $\mathbf{u}_1,\dotsc,\mathbf{u}_s\in \rend^N$  such that  $||\mathbf{u}_i||$ is the $i$-th successive minimum of the  $K$-lattice $\Lambda=\langle \mathbf{u}_1,\dotsc, \mathbf{u}_s\rangle_{\mathcal{O}_K}$.   Then 
\begin{equation*}
h_2(H+P)\leq C_1(N,s)\prod_{i=1}^s ||\mathbf{u}_i||^2\left(\frac{2^N|D_K|^{\frac{N}{2}}}{\omega_{2(N-s)}\omega_{2s}} \sum_{i=1}^s\frac{\hat h(\mathbf{u}_i(P))}{||\mathbf{u}_i||^2}+\ccinque(E)\right).
\end{equation*}
where 
$$C_1(N,s)= N(N-s+1)3^N N! \left(12^{N-1}2\right)^s,$$
$\omega_{2n}=\pi^n/n!$  and $\ccinque(E)$ is defined in Proposition \ref{confrontoaltezze}.
\end{propo}
\begin{proof}

%Let $H$ be a component of the  algebraic subgroup defined by the $s\times N$ matrix with rows $u_1,\ldots,u_s\in\rend^N$. 
%Let $\Lambda\in K^N$ be the lattice associated with $H$, as described in Section \ref{prelim-S}, and $\Lambda^\perp$ its orthogonal lattice.

%We set \[u_i=v_i+\tau w_i\] for every $i=1,\ldots, s$ where $v_i,w_i\in \Z^N$. 

 %Let $\Lambda_v\subset  \R^N$ be lattice associated with the matrix with rows $v_1,\ldots,v_s$ and $\Lambda_w\subset  \R^N$ the one associated with the matrix with rows $w_1,\ldots,w_s$. Let $\Lambda_v^\perp$ and $\Lambda_w^\perp$ be their orthogonal lattices, respectively.
Let $\Lambda^\perp$ be the orthogonal  $K$-lattice of $\Lambda$ and let $\mathbf{u}_{s+1},\dotsc,\mathbf{u}_N$ be generators  of $\Lambda^\perp$, such that $||\mathbf{u}_{s+1}||,\dotsc,||{\mathbf{u}_N}||$ are the successive minima of $\Lambda^\perp$.

The $(N-s)\times N$ matrix with rows $\mathbf{u}_{s+1},\dotsc,\mathbf{u}_N$ defines an  algebraic subgroup $H^\perp$, and for any point $P\in E^N$ there are  two points $P_0\in H$, $P^\perp\in H^\perp$, unique up to torsion points in $H\cap H^\perp$, such that $P=P_0+P^\perp$.

Let $U$ be the $N\times N$ matrix with rows $\mathbf{u}_1,\dotsc,\mathbf{u}_N$, and let $\Delta$ be its determinant.

Notice that
\[
 |\Delta|=\det\Lambda \cdot \det \Lambda^\perp
\]
because $\Lambda$ and $\Lambda^\perp$ are orthogonal.

We remark that $\mathbf{u}_i(P_0)=0$ for all $i=1,\dotsc,s$, because $P_0\in H$, and $\mathbf{u}_i(P^\perp)=0$ for all $i=s+1,\dotsc,N$ because $P^\perp\in H^\perp$.

Therefore

\begin{equation*}
UP^\perp=
\left(
\begin{array}{c}
\mathbf{u}_1(P^\perp)\\
\vdots\\
\mathbf{u}_s(P^\perp)\\
0\\
\vdots\\
0
\end{array}
\right)
=
\left(
\begin{array}{c}
\mathbf{u}_1(P_0+P^\perp)\\
\vdots\\
\mathbf{u}_s(P_0+P^\perp)\\
0\\
\vdots\\
0
\end{array}
\right)
=
\left(
\begin{array}{c}
\mathbf{u}_1(P)\\
\vdots\\
\mathbf{u}_s(P)\\
0\\
\vdots\\
0
\end{array}
\right).
\end{equation*}
Recall, that there exists a matrix  $U^*$, called the adjugate matrix of $U$, with coefficients in $\rend$, such that 
$
UU^*=U^*U=(\det U)\mathrm{Id}.
$
So\begin{equation}\label{matri}
\Delta P^\perp=U^*UP^\perp=U^*\left(
\begin{array}{c}
\mathbf{u}_1(P)\\
\vdots\\
\mathbf{u}_s(P)\\
0\\
\vdots\\
0
\end{array}
\right).
\end{equation}
 
Moreover by   Hadamard's inequality, the $i$-th column of $U^*$ has all entries with absolute value bounded by \begin{equation}\label{boundadj}
\frac{||{\mathbf{u}_1}||\dotsm ||{\mathbf{u}_N}||}{||{\mathbf{u}_i}||}.\end{equation}

Recall that for  $\alpha\in \rend$ and $Q\in E$ then $\hat h(\alpha Q)=\abs{\alpha}^2 \hat h(Q)$, because $\deg \alpha=|\alpha|^2.$ 
Therefore, computing canonical heights in \eqref{matri} and using \eqref{boundadj}  we have 
\begin{equation*}
\hat h(P^\perp)\leq  \frac{N||{\mathbf{u}_1}||^2\dotsm ||{\mathbf{u}_N}||^2}{|{\Delta}|^2}\sum_{i=1}^s\frac{\hat h(\mathbf{u}_i(P))}{||{\mathbf{u}_i}||^2}.
\end{equation*}

Recall inequality \eqref{mu2mu^}, which gives
 \begin{equation*}
\mu_2(H+P)\leq \hat\mu(H+P) + N\ccinque(E)\end{equation*}
where  $\ccinque(E)$ was defined in Proposition \ref{confrontoaltezze}.
By \cite{preprintPhilippon} we have
$
\hat\mu(H+P)=\hat h(P^\perp)
$
and therefore, applying Zhang's inequality and (\ref{mu2mu^}) we obtain
\begin{align}
\notag h_2(H+P)&\leq (N-s+1)(\deg H)(\hat h(P^\perp)  +N\ccinque(E))\leq\\
\label{ultimarigacost}&\leq N (N-s+1)(\deg H) \left(\frac{  ||{\mathbf{u}_1}||^2\dotsm ||{\mathbf{u}_N}||^2}{{\abs{\Delta}^2}}\sum_{i=1}^s\frac{\hat h(\mathbf{u}_i(P))}{||{\mathbf{u}_i}||^2}+\ccinque(E)\right).
\end{align}

By \eqref{gradorighe}  we know
\[
 \deg H\leq 3^N N! \left(12^{N-1} 2\right)^s\prod_{i=1}^s ||\mathbf{u}_i||^2.
\]
 Using  Minkowski's theorem \ref{Minko} we deduce 

\[
\frac{\deg H}{(\det\Lambda)^2}\leq 3^N N! (12^{N-1} 2)^s\frac{2^s|D_K|^{\frac{s}{2}}}{\omega_{2s}},
\]
and
\[
 \frac{\prod_{i=s+1}^N ||\mathbf{u}_i||^2}{(\det\Lambda^\perp)^2}\leq \frac{2^{N-s}|D_K|^{\frac{N-s}{2}}}{\omega_{2(N-s)}}.
\]

Plugging these inequalities in \eqref{ultimarigacost} and using $|\Delta|=\det \Lambda \det \Lambda^\perp$, we obtain the desired bound of the statement.
%\begin{equation}\label{baltez}
 %h_2(H+P)\leq N(N-s+1)3^N N! \left(12^{N-1}2\right)^s\prod_{i=1}^s ||\mathbf{u}_i||^2
%\left(\frac{4^N|D_K|^{\frac{N}{2}}}{\omega_{2(N-s)}\omega_{2s}} \sum_{i=1}^s\frac{\hat h(\mathbf{u}_i(P))}{||\mathbf{u}_i||^2}+\ccinque(E)\right).
%\end{equation}
\end{proof}

%Now if $E$ is non CM, using the classical Minkowski second theorem and the sharper estimate \ref{gradononCM} for $\deg H$, following te same proof of \cite{ExpTAC} we have
%\begin{equation*}
 %h_2(H+P)\leq N(N-s+1)3^N N! \left(12^{N-1}\frac{3}{2}\right)^s\prod_{i=1}^s ||u_i||^2\left(\frac{4^N}{(\omega_{N-s}\omega_{s})^2} \sum_{i=1}^s\frac{\hat h(u_i(P))}{||u_i||^2}+\ccinque(E)\right).
%\end{equation*}
%(notice that the function $\hat h$ we use here is 3 times the one we used in \cite{ExpTAC}).
%Since $\omega_{2(N-s)}\omega_{2s}<(\omega_{N-s}\omega_{s})^2$ (this follows from the fact that $\Gamma(n+1)>\Gamma(n/2+1)^2$ for all positive integers $n$, which can be proved using Jensen's inequality),  formula \eqref{baltez} is actually a bound for $h(H+P)$ both in the CM and in the non CM case.

\section{The auxiliary translate}
\label{atrans}
In this section we construct the auxiliary translate which  approximates the torsion variety passing through a point. We shall estimate the invariants of this  translate explicitly. In  \cite[Proposition 3.1]{ExpTAC} we did it for the non CM case. In the following proposition we generalize the result also to the CM case. 
\begin{propo}\label{main} 
Let $E$ be an elliptic curve with CM by $\tau$ and let $K=\Q(\tau)$  with discriminant $D_K$.
Let $P=(P_1,\dotsc,P_N)\in B\subset  E^N$, where $B$ is a torsion variety of dimension $m$.
Let $T\geq 1$ be a real number.

Then there exists an abelian subvariety $H\subset  E^N$ of codimension $s$ such that 
\begin{align*}
\deg(H+P)\leq&  C_3(N,s)  |D_K|^sT\\
h_2(H+P)\leq&C_4(N,s,m) |D_K|^{\frac{N}{2}+s+2}T^{1-\frac{N}{ms}}\hat h(P)+ C_5(N,s,E) |D_K|^sT 
\end{align*}
where
\[C_3(N,s)=3^N N! \left(12^{N-1}2\right)^s,\]
\[C_4(N,s,m)=\frac{C_1(N,s)C_2(m)N^2 s 2^{N+2} (2(2N)^{2N})^{\frac{2}{m}}}{\omega_{2(N-s)}\omega_{2s}}\]
\[C_5(N,s,E)= C_1(N,s)\ccinque(E)\]
where $C_1(N,s)= N(N-s+1)3^N N! \left(12^{N-1}2\right)^s$,  $C_2(m)=\frac{m^3 (2m!)^4}{2^{4m-5}}$, $\omega_{2n}=\pi^n/n!$ and  $\ccinque(E)$ is given in Proposition \ref{confrontoaltezze}.
\end{propo}
To extend to an imaginary quadratic field $K$ what done in \cite{ExpTAC} for rational numbers, we choose a $\Z$-basis  of the ring of integers $\mathcal{O}_K$ of $K$ which gives the successive minima of $\mathcal{O}_K$ as $\Z$-lattice. Via this identification, we associate to any object over $\mathcal{O}_K$ an object over $\Z$, in a way that bounds are  kept small. The estimates in the CM case are more elaborated and they are carried out in the next paragraphs, which are devoted to the proof of  this proposition.

\subsection{Geometry of numbers}\label{subsec-geomnum}
%In all the section we assume that $E$ has CM by $\tau$. 
If $u$ is a vector, we denote by $||u||$ its euclidean norm and for a linear form $L$ we denote by $||L||$  the norm of the vector of the coefficients of $L$.
We have the following CM variant of Lemma 7.5 of \cite{ExpTAC}.
\begin{lem}\label{LemmaHab1}
Let $E$ be an elliptic curve with CM by $\tau$, let $K=\Q(\tau)$ and $D_K$ its discriminant. Let $P=(P_1,\dotsc,P_N)\in B\subset  E^N$, where $B$ is a torsion variety of dimension $ m$.
Then there exist linear forms $\mathbf{L}_1,\dotsc,\mathbf{L}_{m}\in \mathbb{C}[X_1,\dotsc,X_N]$ such that $||\mathbf{L}_j||\leq 1/|D_K|$  and 
\begin{equation*}
\hat h(t_1P_1+\dotsb+t_NP_N)\leq C_2(m) N^2 |D_K|^2\max_{1\leq j\leq m}\{|\mathbf{L}_j(\mathbf{t})|^2\}\hat h (P)
\end{equation*}
for all $\mathbf{t}=(t_1,\ldots,t_N)\in\rend^N$, where 
\[C_2(m)=\frac{m^3 (2m!)^4}{2^{4m-5}}.\]
\end{lem}

\begin{proof}
 For $Q,R\in E$, consider the pairing defined by Philippon in \cite{preprintPhilippon}  
\begin{equation}\label{hpairing}\langle Q,R\rangle=\langle Q,R\rangle_{NT}-\frac{1}{\sqrt D} \langle Q,\sqrt D R\rangle_{NT}\end{equation}
where $\langle,\rangle_{NT}$ is the N\'eron-Tate pairing and $D$ is a squarefree negative integer such that $K=\qe(\sqrt D)$.

This pairing gives the N\'eron Tate height since $\langle Q, Q\rangle=\hat h(Q)$, and it is sesquilinear  and hermitian.

Since $P$ belongs to an algebraic subgroup of dimension $m$, then the group $\Gamma_P$ generated by its coordinates has rank   $2m$ (over $\Z$). By Lemma 3 of \cite{ioannali} with $\Gamma=\Gamma_P$ and $r=2m$, there exist generators $g_1,\ldots,g_{2m}  \in E$ of the free part of $\Gamma_P$  such that for every integer $a_i$ we have
\begin{equation}\label{EveAnnali}\hat h(a_1g_1+\ldots+a_m g_{2m})\geq \frac{2^{4m-2}}{(2m)^2 (2m)!^4}\left(\sum_{i=1}^{2m} |a_i|^2 \hat h(g_i) \right).\end{equation}
We now extract from the $g_i$ an $\rend$-basis $\mathbf{g}_1,\dots \mathbf{g}_m$ of the free part of $\Gamma_P$.  We want to lower bound $\hat h(\alpha_1\mathbf{g}_1+\ldots+\alpha_m \mathbf{g}_{m})$ for every $\alpha_i\in \rend$.  Writing  $\alpha_i=a_i+\tau b_i\in \rend$, where $a_i,b_i\in \Z$,  we have
\begin{align*}
\hat h(\alpha_1\mathbf{g}_1+\ldots+\alpha_m \mathbf{g}_{m})&=\hat h(\sum_{i=1}^m a_i \mathbf{g}_i +\tau \sum_{i=1}^m b_i \mathbf{g}_i)=\hat h(\sum_{i=1}^m a_i \mathbf{g}_i)+|\tau|^2  \hat h(\sum_{i=1}^m b_i \mathbf{g}_i)+\\
&+\langle \sum_{i=1}^m a_i \mathbf{g}_i,\tau \sum_{i=1}^m b_i \mathbf{g}_i\rangle+\langle \tau \sum_{i=1}^m b_i \mathbf{g}_i,\sum_{i=1}^m a_i \mathbf{g}_i\rangle.\\
%=& \hat h(\sum_{i=1}^m a_i \mathbf{g}_i)+|\tau|^2  \hat h(\sum_{i=1}^m b_i \mathbf{g}_i)+2\mathrm{Re}\left(\langle \sum_{i=1}^m a_i \mathbf{g}_i,\tau \sum_{i=1}^m b_i \mathbf{g}_i\rangle\right).
\end{align*}

Using the definition \eqref{hpairing}  and the fact that the pairing is hermitian, we have \begin{align*}\langle \sum_{i=1}^m a_i \mathbf{g}_i,\tau \sum_{i=1}^m b_i \mathbf{g}_i\rangle+\langle \tau \sum_{i=1}^m b_i \mathbf{g}_i,\sum_{i=1}^m a_i \mathbf{g}_i\rangle=&2\mathrm{Re}\left(\langle \sum_{i=1}^m a_i \mathbf{g}_i,\tau \sum_{i=1}^m b_i \mathbf{g}_i\rangle\right)=\\
=&2\langle \sum_{i=1}^m a_i \mathbf{g}_i,\tau \sum_{i=1}^m b_i \mathbf{g}_i,\rangle_{NT};\end{align*}
on the other hand, by the sesquilinearity we also obtain
\[ \mathrm{Re}\left(\langle \sum_{i=1}^m a_i \mathbf{g}_i,\tau \sum_{i=1}^m b_i \mathbf{g}_i\rangle\right)=\mathrm{Re}\left(\tau \langle \sum_{i=1}^m a_i \mathbf{g}_i, \sum_{i=1}^m b_i \mathbf{g}_i\rangle\right)\]
thus \begin{equation}\label{ReugualeNT}\mathrm{Re}\left(\tau \langle \sum_{i=1}^m a_i \mathbf{g}_i, \sum_{i=1}^m b_i \mathbf{g}_i\rangle\right)=\langle \sum_{i=1}^m a_i \mathbf{g}_i,\tau \sum_{i=1}^m b_i \mathbf{g}_i,\rangle_{NT}.\end{equation}
If $\tau=\sqrt{D}$ or $2\sqrt{D}$, by \eqref{ReugualeNT} it easily follows \[\langle \sum_{i=1}^m a_i \mathbf{g}_i,\tau \sum_{i=1}^m b_i \mathbf{g}_i,\rangle_{NT}=0,\] while  if $\tau=(1+\sqrt{D})/2$ one has \[\langle \sum_{i=1}^m a_i \mathbf{g}_i,\tau \sum_{i=1}^m b_i \mathbf{g}_i,\rangle_{NT}=\frac{\langle \sum_{i=1}^m a_i \mathbf{g}_i, \sum_{i=1}^m b_i \mathbf{g}_i,\rangle_{NT}}{2}.\]
By  \eqref{EveAnnali} and  $|\tau|^2\geq 1$ we finally obtain 
\begin{equation}
\label{sesq}
\hat h(\alpha_1\mathbf{g}_1+\ldots+\alpha_m \mathbf{g}_{m})\geq \frac{2^{4m-3}}{(2m)^2 (2m)!^4}\left(\sum_{i=1}^{m} |\alpha_i|^2 \hat h(\mathbf{g}_i) \right).
\end{equation}

By the choice of the $\mathbf{g}_i$, for every $1\leq i\leq N$ we can write
\[P_i=\zeta_i+\gamma_{i1}\mathbf{g}_1+\ldots+\gamma_{im}\mathbf{g}_m\]
for some $\gamma_{ij}\in \rend$  and  torsion points $\zeta_i\in E$.

Let $A=\max_{i,j}\{|D_K|^2|\gamma_{ij}|^2\hat h(\mathbf{g}_j)\}$. We can suppose $A>0$, otherwise $P$ would be a torsion point and the lemma would be trivial.
Define now:
\begin{align*}
\tilde{L}_j&=\gamma_{1j}X_1+\dotsb+\gamma_{Nj}X_N &j&=1,\dotsc,m\\
L_j&=\left(\frac{\hat h(\mathbf{g}_j)}{NA}\right)^\frac{1}{2}\tilde{L}_j &j&=1,\dotsc,m.
\end{align*}
By the definition of $A$ we obtain $||L_j||\leq 1/|D_K|$. Moreover for every $\mathbf{t}=(t_1,\ldots,t_N)\in\rend^N$, we have
\begin{equation*}
t_1P_1+\dotsb+t_NP_N=\xi+\sum_{i=1}^m \tilde{L}_j(\mathbf{t})\mathbf{g}_j
\end{equation*}
where $\xi$ is a  torsion point.
Computing heights we get 
\begin{align}
\notag \hat h(t_1P_1+\dotsb+t_NP_N)=\hat h\left(\sum_{j=1}^m \tilde{L}_j(\mathbf{t})g_j\right)\leq N\sum_{j=1}^m|\tilde{L}_j(\mathbf{t})|^2 \hat h(g_j)=\\
\label{ineqlemma}= N^2A\sum_{j=1}^m|{L}_j(\mathbf{t})|^2\leq mN^2A\max_{1\leq j\leq m}\{|{L}_j(\mathbf{t})|^2\}.
\end{align}

If $i_0,j_0$ are the indices for which the maximum is attained in the definition of $A$, then by (\ref{sesq}) we get 
\begin{equation*}
\frac{2^{4m-3}}{(2m)^2 (2m)!^4}A=\frac{2^{4m-3}}{(2m)^2 (2m)!^4}|D_K|^2|\gamma_{i_0 j_0}|^2\hat h(\mathbf{g}_{j_0})\leq |D_K|^2\hat h(P_{i_0})\leq  |D_K|^2\hat h(P).
\end{equation*}
Combining this with inequality \eqref{ineqlemma}, we conclude the proof.

\end{proof}

The following lemma is proven  by a simple computation  with complex numbers, and it is useful to decompose linear forms over $\mathbb{C}$ in linear forms over $\mathbb{R}$.

\begin{lem}\label{dec} Let $E$ be an elliptic curve with CM by $\tau$. Let $K=\Q(\tau)$ and $D_K$ its discriminant.
Let $\mathbf{L}\in  \mathbb{C} [X_1,\dotsc,X_N]$ be a linear form and let $\mathbf{t}\in \rend^N$ then
$$\mathbf{L}(\mathbf{t})={L}_1(t_1)+{L}_2x_0(t_2)+({L}_2(t_1)+{L}_1(t_2)+{L}_2y_0(t_2))\tau$$
where $\mathbf{L}={L}_1+L_2\tau$ with $L_i\in  \mathbb{R} [X_1,\dotsc,X_N]$, $\mathbf{t}=t_1+t_2\tau$ with ${t}_i\in \mathbb{Z}^N$ and $\tau^2=x_0+y_0\tau$.

Moreover $$||(L_1,L_2x_0)|||\le |D_K|||L||$$ and $$||(L_2,L_1+L_2y_0)||\le 2||L||;\,\,\,\,\,||(L_2,L_1+L_2y_0)||\le |D_K| ||L||.$$
\end{lem}
\begin{proof} The first part is a simple computation implied from the linearity and from the rule of multiplication of complex numbers, namely for $a_i, b_i \in \mathbb{R}$ we have  $(a_1+a_2\tau)(b_1+b_2\tau)=a_1b_1+a_2x_0b_2+(a_1b_2+a_2b_1+a_2y_0b_2)\tau$. 

For the second part one shall  do the computations with the three different possible cases for the values of  $\tau$ and $\tau^2$ given in  Section \ref{CM}, and compare them with the  value of the discriminant when $D$ is congruent $1$ modulo $4$ or not.
For the  last two inequalities one shall observe that for $D=-1$ then $y_0=0$. For $D<-1$ the inequalities are immediate because $y_0$ is either $0$ or $1$ and $|D_K|\ge2$.
\end{proof}

We can now prove a lemma of geometry of numbers which relies on   Lemma 2 of \cite{hab} and on a decompositions of forms over $\mathbb{C}$ in forms over $\mathbb{R}$.
\begin{lem}\label{LemmaHab3} 
 Let $E$ be an elliptic curve with CM by $\tau$ and $K=\qe(\tau)$  with discriminant $D_K$. Let $1\leq m \leq N$ and let $\mathbf{L}_1,\dotsc, \mathbf{L}_{m}\in  \mathbb{C} [X_1,\dotsc,X_N]$ be linear forms with $||{\mathbf{L}_j}||\leq 1/|D_K|$ for all $j$. Then for any real number $T\geq 1$ and any integer $s$ with $1\leq s\leq N$ there exist linearly independent vectors $\mathbf{u}_1,\dotsc, \mathbf{u}_s\in\rend^N$ such that  \[||{\mathbf{u}_1}||\dotsm||{\mathbf{u}_s}||\leq |\tau|^{s}T\] and for $1\leq j\leq m$ and $1\leq k\leq s$  
\[ ||{\mathbf{u}_1}||\dotsm ||{\mathbf{u}_s}||\frac{\abs{\mathbf{L}_j(\mathbf{u}_k)}}{||{\mathbf{u}_k}||}\leq  2(2(2N)^{2N})^{\frac{1}{m}}|\tau|^{s}T^{1-\frac{N}{ms}}
.\]
\end{lem}
\begin{proof}
 If $T\le (2(2N)^{2N})^\frac{s}{N}$ then $(2(2N)^{2N})^{\frac{1}{m}}T^{1-\frac{N}{ms}}\geq 1$. It is then sufficient to take $u_i$ to be $s$ elements  of the standard basis of $\mathbb Z^N$. \\

We can then assume that $T> (2(2N)^{2N})^\frac{s}{N}$.

We shall define new forms  in $\mathbb{R}[X_1,\dots, X_N,\tau X_1,\dots,\tau X_N]$ that will allow us to use Habegger \cite[Lemma 2]{hab} and to produce the wished  estimates.
Let  $\mathbf{L}_i=L_{1i}+L_{2i}$ with $L_{1i}$  in   $\mathbb{R}[X_1,\dots, X_N]$ and $L_{2i}$ in   $\mathbb{R}[\tau X_1,\dots, \tau X_N]$.   In view of Lemma \ref{dec} we define:
\begin{equation}
\begin{split}
L_i&=(L_{1i},L_{2i}x_0) \in \mathbb{R}[X_1,\dots, X_N,\tau X_1,\dots,\tau X_N]\\
L'_i&=(L_{2i},L_{2i}y_0+L_{1i})  \in \mathbb{R}[X_1,\dots, X_N,\tau X_1,\dots,\tau X_N],
\end{split}
\end{equation}
where $\tau^2=x_0+y_0\tau$.
Then  $||L_i||\le  |D_K|  ||\mathbf{L}_i||\le 1$ and $||L'_i||\le  |D_K| ||\mathbf{L}_i||\le 1$.

We can then apply \cite{hab} Lemma 2  to the $2m$ forms ${L}_1, \dots {L}_m$ and ${L'}_1, \dots ,{L'}_m$  in  $\mathbb{R}[X_1,\dots, X_N,\tau X_1,\dots,\tau X_N]$ with $n=2N$ and his  $m$ twice the $m$ appearing here. So   for any real number $\rho\geq 1$  there exist linearly independent vectors $U_1,\dotsc, U_{2N}\in\Z^{2N}$ and  positive reals $\lambda_1\le\dots \le \lambda_{2N}$ such that  
\begin{equation}
\label{26}||U_i||\le \lambda_i\,\,\,\,\,{\rm and}\,\,\,\,\lambda_1\cdots \lambda_{2N}\le 2(2N)^{2N} \rho ^{2m}\end{equation}
and for every $i\le m$ and $k\le 2N$
\begin{equation}
\label{27}||L_i(U_k)||\le \rho^{-1}\lambda_k \,\,\,\,\,{\rm and}\,\,\,\,\||L'_i(U_k)||\le \rho^{-1}\lambda_k. \end{equation}
 
 For every $1\leq s\leq N$, consider the writing of  the vectors $U_1,\ldots,U_{2s}$ as $U_1=(A_{11},A_{21}),\dots ,U_{2s}=(A_{1,2s},A_{2,2s})$ with $A_{ij}\in \Z^N$ and  set $\mathbf{U}_i= A_{1i}+A_{2i}\tau$. Extract $s$ vectors  $\mathbf{u_i}= a_{1i}+\tau a _{2i}$ from the  $\mathbf{U}_1, \dots , \mathbf{U}_{2s}$ such that the vectors $\mathbf{u}_1, \dots , \mathbf{u}_s$  are $K$-linearly independent and of minimal norm. We denote $u_i=(a_{1i}, a _{2i})$.  
Then
 $||u_i||^2\le ||U_{2i-1}|| ||U_{2i}||$. Moreover, from (\ref{26})
 we deduce
 $$||u_1||^2\dots ||u_{s}||^2\le ||U_1||\dots ||U_{2s}||\le \lambda_1 \dots \lambda_{2s} \le (\lambda_1 \dots \lambda_{2N})^{\frac{s}{N}}\le (2(2N)^{2N})^\frac{2s}{N}\rho^{\frac{2ms}{N}}.$$

 Choose $\rho $ so that $T^2:=(2(2N)^{2N})^\frac{2s}{N}\rho^{\frac{2ms}{N}}$.  Since $T> (2(2N)^{2N})^\frac{s}{N}$ then $\rho\ge1$.
 
 Thus  $$||\mathbf{u}_1||\dots ||\mathbf{u}_{s}||\le |\tau|^s ||u_1||\dots ||u_{s}||\le |\tau|^s T$$
 which proves the first bound in the statement.

 By (\ref{27}) and from the fact that $u_k$ is one vector among the  $U_j$ with $j\le 2k-1$ we get that  
 $$||L_i(u_k)||^2=||L_i(U_j)||^2\le \rho^{-2}\lambda_j^2\le \rho^{-2}\lambda_{2k-1}\lambda_{2k}$$
 and similarly
 $$||L'_i(u_k)||^2\le \rho^{-2}\lambda_{2k-1}\lambda_{2k}.$$
This together with (\ref{26}) gives that 
  for all $i\le m$ we have 
 $$||u_1||^2\dots ||u_{s}||^2\frac{|L_i(u_k)|^2}{||u_k||^2}\le \rho^{-2} \lambda_1 \dots \lambda_{2s} \le \rho^{-2}T^2=(2(2N)^{2N})^{\frac{2}{m}} T^{2(1-\frac{N}{ms})}$$
 and 
 $$||u_1||^2\dots ||u_{s}||^2\frac{|L'_i(u_k)|^2}{||u_k||^2}\le \rho^{-2} \lambda_1 \dots \lambda_{2s} \le \rho^{-2}T^2=(2(2N)^{2N})^{\frac{2}{m}} T^{2(1-\frac{N}{ms})}.$$

By Lemma \ref{dec} we deduce that the $K$-linearly independent vectors  
$$\mathbf{u_i}= a_{1i}+\tau a _{2i}\in \rend^N\,\,\,\,\,\,{\rm with}\,\,\,\,\, u_i=(a_{1i}, a _{2i})$$ satisfy 
\begin{align*}||\mathbf{u}_1||\dots ||\mathbf{u}_{s}||\frac{|\mathbf{L}_i(\mathbf{u}_k)|}{||\mathbf{u}_k||}\le &
|\tau|^{s-1}||{u}_1||\dots ||{u}_{s}||\left(\frac{|{L}_i({u}_k)|+|\tau| |{L'}_i({u}_k)|}{||u_k||}\right)\\\le& 2 |\tau|^{s}(2(2N)^2N)^{\frac{1}{m}} T^{1-\frac{N}{ms}}.\end{align*}
This concludes the proof.
\end{proof}
%{\bf (NOTA: il bund che avevmo nel caso non CM era
%\[||{u_1}||\dotsm||{u_s}||\leq T\] e
%\[||{u_1}||\dotsm ||{u_s}||\frac{L_j(u_k)}{||{u_k}||}\leq  \frac{2^N}{\omega_N}\binom{N+m}{N}^{1/2} T^{1-\frac{N}{ms}}\]
%mentre nel caso CM veniva \[||{u_1}||\dotsm||{u_s}||\leq C(N,m)^\frac{s}{N}T\] e 
%\[||{u_1}||\dotsm ||{u_s}||\frac{L_j(u_k)}{||{u_k}||}\leq  C(N,m)^{\frac{N+sm}{Nm}} T^{1-\frac{N}{ms}} \]
%e il bound nel teorema e' maggiore dei due (la constante $C(N,m)\geq 1$ e si vede facilmente che $ \frac{2^N}{\omega_N}\binom{N+m}{N}^{1/2}\leq C(N,m)\leq C(N,m)^{(m+1)/m}$, usando che $\omega_N=\frac{\pi^{(N-1)/2} 2^{(N+1)/2)}}{N!!}$). Se vuoi metto due righe di dimostrazione)}

 We are now ready to prove Proposition \ref{main}.

\begin{proof}[Proof of  Proposition \ref{main}]
 If $E$  is non  CM then this is proven in   \cite{ExpTAC} Proposition 3.1. Note that we used there another normalization for the   height functions (namely, the $\hat h$ used in the present paper is 3 times the one used in \cite{ExpTAC}), however the bounds in that proposition are sharper than  those presented here which  hold in both cases.

Assume now that $E$ has CM. Let $\mathbf{L}_1,\ldots,\mathbf{L}_{m}\in K[X_1,\ldots,X_N]$ be the linear forms constructed in Lemma \ref{LemmaHab1}.  Then
\begin{equation}
\label{uno}
\hat h(\mathbf{u}(P)) \leq  C_2(m) |D_K|^2 N^2 \max_{1\leq j\leq m}\{||\mathbf{L}_j({\mathbf{u}_k})||^2\}\hat h (P)
\end{equation} for all $\mathbf{u}\in \rend^N$.

By Lemma \ref{LemmaHab3} applied with $\sqrt{T}$,  there exist $s$ vectors $\mathbf{u}_1,\ldots,\mathbf{u}_s\in\rend^N$ such that 
\begin{equation}||\mathbf{u}_1||^2\cdots ||\mathbf{u}_s||^2 \leq  |\tau|^{2s} T\end{equation} and 
\begin{equation} ||{\mathbf{u}_1}||^2\dotsm ||{\mathbf{u}_s}||^2\frac{||{\mathbf{L}_j(\mathbf{u}_k)}||^2}{||{\mathbf{u}_k}||^2}\leq  4|\tau|^{2s}(2(2N)^{2N})^{\frac{2}{m}} T^{1-\frac{N}{ms}}
%||{\mathbf{u}_1}||\dotsm ||{\mathbf{u}_s}||\frac{\mathbf{L}_j(\mathbf{u}_k)}{||{\mathbf{u}_k}||}\leq  C(N,m)^{\frac{N+sm}{Nm}} T^{1-\frac{N}{ms}}
\end{equation}

and from (\ref{uno}) we have 
\begin{equation}
\hat h(\mathbf{u}_k(P)) \leq  C_2(m) |D_K|^2 N^2 \max_{1\leq j\leq m}\{||\mathbf{L}_j({\mathbf{u}_k})||^2\}\hat h (P).
\end{equation}

Consider the  algebraic subgroup defined by the equations $\mathbf{u}_1(\mathbf{X})=0,\ldots,\mathbf{u}_s(\mathbf{X})=0$ and call $H$ the irreducible component containing $0$. Since the $\mathbf{u}_i$ are $K$-linearly independent the codimension of $H$ is $s$. Moreover, by \eqref{gradorighe} its degree is bounded as
\[\deg(H+P)\leq  \left(3^N N! (12^{N-1}2)^s \right) |\tau|^{2s} T.\]
This proves the first bound because $|\tau|^2\le |D_K|$.

Substituting the above estimates in  Proposition \ref{stimaaltezzanormalizzatatraslato}, we obtain that
\begin{align*}
 h_2(H+P)\leq&\frac{C_1(N,s)C_2(m)N^2 s 2^{N+2} (2(2N)^{2N})^{\frac{2}{m}}}{\omega_{2(N-s)}\omega_{2s} } |D_K|^{\frac{N}{2}+s+2}T^{1-\frac{N}{ms}}\hat h(P)+\\
&
+ C_1(N,s)\ccinque(E) |D_K|^s T.
\end{align*}
% Substituting the values of  the constants $C(N,m)$ and $c(m)$, computed in Lemma \ref{LemmaHab3}  and Lemma  \ref{LemmaHab1} respectively, we get the bounds in the statement.

\end{proof}

\section{The proof of the main theorem}\label{DIMO}

In this section we prove our main theorem. We divide the theorem into two parts. The first is the case of transverse curves, the second is the general case.
At the end of the section we also prove  Corollaries \ref{caso_poly} and \ref{cardinalita}.
\subsection{Transverse curves} 
\begin{thm}\label{casotras}
Let $\Ci$ be a curve  transverse in $E^N$,  let  $K=\mathrm{Frac}(\rend)$ with discriminant $D_K$. Then the set of points of $\Ci$ of rank $r\le N-1$ has N\'eron-Tate height bounded as 
\begin{align*}
 \hat{h}(P)\leq& |D_K|^{\frac{2N+Nr+4r}{2(N-r)}}\left(c_1(N,r)  h_2(\Ci)(\deg \Ci)^{\frac{r}{N-r}}+c_2(N,r,E) (\deg \Ci)^{{\frac{N}{N-r}}}\right)+N^2\cquattro(E),
\end{align*}
where \begin{align*}
&c_1(N,r)=\frac{2^{\frac{2N^2+3N+Nr-4r^2+7r+2}{N-r}} 3^{\frac{(2N-1)N}{N-r}}}{(\omega_{2(N-1)}\omega_2)^{\frac{r}{N-r}}} \frac{N^{\frac{5N+4r}{N-r}}}{(N-1)^{\frac{r}{N-r}}} 
(N!)^{\frac{N}{N-r}}(2r!)^{\frac{4r}{N-r}}r^{\frac{3r}{N-r}} \\
& c_2(N,r,E)=c_1(N,r)\left(N^2 \ccinque(E)+\frac{H_N+\log2\left(2(3^N-1)-N\right)}{2}\right),
\end{align*}
 $\omega_{2n}=\pi^n/n!$, $H_N$ is the $N$-th harmonic number and $\ccinque(E)$  is the explicit constant defined in Proposition \ref{confrontoaltezze} and it depends only on the coefficients of $E$.
\end{thm}
\begin{proof}
 By assumption, the point $P$ has rank $r\le N-1$, then $P\in B$ where $B$ is a torsion variety with $\dim B= r$.

We apply Proposition \ref{main} to $P$,  with $m=r$, $s=1$ and $T$ as a free parameter that will be specified later. This gives a translate $H+P$ of dimension $N-1$, of degree explicitly bounded in terms of $T$ and  such that $h_2(H+P)$ is explicitly bounded  in terms of $\deg H$ and $\hat h(P)$ and $T$.
More precisely for 
\begin{align}
\label{boalfa}&\alpha=  (N!) 3^{2N-1} 2^{2N-1} |D_K| ,\\
\label{bobeta}&\beta= (N!) N^{\frac{4(N+r)}{r}} (2r)!^4 r^3 \frac{2^{3N-4r+6+\frac{2(2N+1)}{r}}3^{2N-1}}{\omega_{2(N-1)}\omega_2} |D_K|^{\frac{N}{2}+3}, \\
&\gamma= N^2 \ccinque(E) \alpha
\end{align}
the degree and the height of the translate $H+P$  are bounded by  Proposition \ref{main} with $m=r$ and  $s=1$ as
 \begin{equation}\label{bodeg}
 \deg (H+P)\leq \alpha T
\end{equation}
 and 
\begin{equation}\label{boh2}
 h_2(H+P)\leq \beta T^{1-\frac{N}{r}}\hat h(P)+\gamma T.
\end{equation}

We remark that $P$ is a component of $\Ci\cap (H+P)$.
{If not, then $\Ci\subset  H+P$ contradicting the transversality of $\Ci$. }
So, we  apply the Arithmetic B\'ezout  Theorem to $\Ci\cap (H+P)$ to bound the height of $P$. We have 
$$h_2(P) \le h_2(\Ci)\deg H +\deg \Ci h_2(H+P)+\csei\deg \Ci \deg H,$$
where \[\csei=\left(\frac{H_N+\log2\left(2(3^N-1)-N\right)}{2}\right)\] is given in \eqref{costaBez} and $H_N$ denotes the $N$-th harmonic number.
Substituting the estimates from \eqref{bodeg} and \eqref{boh2} above, we obtain

\begin{align}\label{boundpunto51b}
h_2(P)\leq& h_2(\Ci)\alpha T+\deg \Ci \left( \frac{\beta}{ T^{{(N-r)}/r}}\hat h(P)+\gamma T \right)+\\ \notag&+\csei\alpha T \deg \Ci
\end{align}

We now choose 
\begin{equation*}
 T=\left(\frac{N}{N-1}\beta \deg C\right)^{\frac{r}{N-r}},
\end{equation*}
so that the coefficient of $\hat h(P)$ at the right-hand side of \eqref{boundpunto51b} becomes $(N-1)/N<1$.

 Recall that by Proposition \ref{confrontoaltezze}  we have
\[\hat{h}(P)\leq h_2(P)+N\cquattro(E)\]
where 
$$C(E)=\frac{h_W(\Delta)+3h(j)}{4}+\frac{h_W(A)+h_W(B)}{2}+4$$
with  $A$ and $B$ the  coefficients of the  Weierstrass equation  for $E$ and $\Delta$ and $j$  the discriminant and the $j$-invariant of $E$.        
 Then we get
\begin{align*}
\hat{h}(P)\leq& \frac{(N-1)}{N} \hat{h}(P)+h_2(\Ci)\alpha T+\deg \Ci \gamma T+\csei\alpha T \deg \Ci+N\cquattro(E)
\end{align*}
and hence 
\begin{align*}
 \hat{h}(P)\leq& N h_2(\Ci)\alpha T+N(\gamma+ \csei\alpha)T \deg \Ci+N^2\cquattro(E)\leq \\
 \leq& \delta h_2(\Ci) (\deg \Ci)^{\frac{r}{N-r}}+\delta (N^2 \ccinque(E)+ \csei) (\deg \Ci)^{\frac{N}{N-r}}+\\
&+N^2\cquattro(E)
\end{align*} 
where $$\delta=N \alpha \left(\frac{N}{N-1} \beta\right)^{\frac{r}{N-r}}.$$
Substituting the values of $\alpha$ from \eqref{boalfa} and $\beta$ from \eqref{bobeta} we get  the desired bound for $\hat{h}(P)$. This concludes the proof.

\end{proof}

\subsection{The general case}

 We can now conclude the proof of Theorem \ref{teorema}, reducing the case of a general curve to the  transverse case with a geometric induction.

\begin{thm}\label{thmgen} 
Let $\Ci$ be a curve of genus at least $2$ embedded in $E^N$ and let $K=\mathrm{Frac}(\rend)$  with discriminant $D_K$. Then the set of points  of $\Ci$ of rank $r< \max(r_\Ci-t_\Ci, t_\Ci)$ has  N\'eron-Tate height bounded as 
\begin{align*}
\hat{h}(P)\leq& (N-r) |D_K|^{\frac{2t_\Ci+t_\Ci r+4r}{2(t_\Ci-r)}}\left(c_1(t_\Ci,r)  h_2(\Ci)(\deg \Ci)^{\frac{r}{t_\Ci-r}}+c_2(t_\Ci,r,E)(\deg \Ci)^{{\frac{t_\Ci}{t_\Ci-r}}}\right)+\\
&+(N-r)t_\Ci^2\cquattro(E)+\frac{h(\Ci)}{\deg \Ci},
\end{align*}
where the constants $c_1(t_\Ci,r)$ and $c_2(t_\Ci,r,E)$ are computed in Theorem \ref{casotras}  and  $\cquattro(E)$ is defined in Proposition \ref{confrontoaltezze}.

\end{thm}

\begin{proof}
 Let $H_0+Q$ be a translate of minimal dimension $t_\Ci$ containing $\Ci$, with $H_0$ an abelian subvariety of $E^N$ and $Q$ a point in $H^{\perp}_0$. Then $\Ci=\Ci_0+Q$ with $\Ci_0$ transverse in $H_0$ and so the rank of any point of $\Ci$ is at least $\rank Q$. Thus there are no points on $\Ci$ of rank $<\rank Q$.  Moreover $\rank Q=r_\Ci-t_\Ci$. Indeed if $B_Q$ is the torsion variety of minimal dimension containing $Q$, then by definition of rank $\dim B_Q=\rank Q$. Thus $\Ci$ is contained in the torsion variety $H_0+B_Q$ of dimension $t_\Ci+\rank Q$. On the other hand, if $B$ is the torsion variety of minimal dimension containing $\Ci$, then $\Ci-Q=\Ci_0\subset  B-Q$. Thus $H_0\subset  B-Q$ and $H_0+Q\subset  B$. Since $H_0$ is an abelian variety then $Q\in B$. Thus $H_0+B_Q\subset  B$. And equality follows by the minimality of $B$. So $r_\Ci=t_\Ci+\rank Q$.
 
 In conclusion there are no points of rank $<\rank Q=r_\Ci-t_\Ci$. This solves the case $r_\Ci-t_\Ci \ge t_\Ci$.\\
 
  We now assume $t_\Ci > r_\Ci-t_\Ci$. The torsion variety $H_0$ is a component of the kernel of a morphism  given by an $(N-t_\Ci) \times N$-matrix $\Phi$ with entries in $\rend$ and of rank $N-t_\Ci$, as described in Section \ref{prelim-S}.  Then, there exists an orthogonal complement matrix $\Phi^\perp$ with entries in ${\rm End}(E)$. Let $P\in \Ci$ be a point of rank $r$, then $P=P_0+Q$ with $P_0\in H_0$ and $\hat{h}(P)=\hat{h}(P_0+Q)=\hat{h}(P_0)+\hat{h}(Q)$ by \cite{preprintPhilippon}. If $P_0$ is a torsion point then we can directly conclude, indeed $\hat{h}(P)=\hat{h}(Q) \le h(\Ci)/\deg \Ci$ by Zhang's inequality.
We may then assume that $P_0$ is  non-torsion. Let $P_{i_0}$ be the coordinate of $P_0$ of maximal height, then $P_{i_0}$ is non-torsion.

Using basic  arguments of linear algebra and a Gauss-reduction process  (which might also determine a  reordering  of the coordinates) we can reduce  the matrix $\Phi$  to the form
 \begin{equation*}
M=
 \left(\begin{array}{cccccc}
a_1&\dots &0&a_{1,N-t_\Ci}&\dots &a_{1,N}\\
\vdots & & \vdots & \vdots& & \vdots\\
0&\dots &a_{N-t_\Ci}&a_{N-t_\Ci,N-t_\Ci+1}&\dots &a_{N-t_\Ci,N}
\end{array}\right),
\end{equation*} 
with $a_i\not=0$ for all $i$ and  the matrix  $\Phi^{\perp}$  to the form
 \begin{equation*}
M^\perp=
 \left(\begin{array}{cccccc}
b_{1,1}&\dots &b_{1,t_\Ci}&b_1&\dots &0\\
\vdots & & \vdots & \vdots& & \vdots\\
b_{t_\Ci, 1}&\dots &b_{t_\Ci,t_\Ci}&0&\dots &b_{t_\Ci}
\end{array}\right),
\end{equation*} 
with $b_i\not=0$ for all $i$.

In addition we can assume that the  $(N-t_\Ci+1)$-th column of the reduced matrices $M$ and $M^{\perp}$ corresponds to the $i_0$ coordinate of the starting order of the factors in $E^N$. Suppose not, then if the $i_0$ coordinate corresponds to one of the last $t_\Ci$  columns it is sufficient  a  reordering of the columns (which corresponds to a permutation of the factors of $E^N$) and of the rows of $M^\perp$. Suppose now that the $i_0$-th coordinate corresponds to the column $j\le N-t_\Ci$, then since $P_{i_0}$ is a non torsion point there exists a non-trivial  entry $M_{jk}\not=0$ with  $N-t_\Ci<k\le N$. In addition $b_k\not= 0$ and by the orthogonality assumption $M_{jk}^\perp\not=0$.  Thus we  can   exchange the $j$-th and $k$-th columns of $M$ and $M^\perp$ and operate a new Gauss reduction to obtain the  matrix $M$ and $M^\perp$ with the wished properties.

Passing to the Lie algebras, the  kernels of these matrices determine 
 the Lie algebra of  $H_0$ and $ H_0^\perp$ as sub-Lie algebras of $E^N$.  Thus the rows of $M^\perp$ are a basis of the Lie algebra of $H_0$. Since  the last $t_\Ci$ columns of $M^\perp$  form an invertible minor, the projection  of the last $t_\Ci$ coordinates of $Lie(H_0)$  on $\C^{t_\Ci}$ is surjective.  Thus    the projection $\pi: E^N \to E^{t_\Ci}$  on the last $t_\Ci$ coordinates is surjective when restricted to  $H_0$ and consequently $\pi(\Ci)=\pi(\Ci_0)$ is transverse in $E^{t_\Ci}$. 
By  a result of Masser and W\"ustholz \cite{MW90}, Lemma 2.1, we have that $\deg \pi(\Ci)\le \deg \Ci$. 
In addition $h(\pi(\Ci)) \le h(\Ci)$. 

We can now apply Theorem \ref{casotras} to the curve $\pi(\Ci)$ in $E^{t_\Ci}$ obtaining
\begin{align}\label{boundpiP}
 \hat{h}(\pi(P))\leq& |D_K|^{\frac{2t_\Ci+t_\Ci r+4r}{2(t_\Ci-r)}}\left(c_1(t_\Ci,r) h_2(\Ci)(\deg \Ci)^{\frac{r}{t_\Ci-r}}+c_2(t_\Ci,r,E)(\deg \Ci)^{{\frac{t_\Ci}{t_\Ci-r}}}\right)+t_\Ci^2\cquattro(E),
\end{align}
where the constants $c_1(t_\Ci,r)$ and $c_2(t_\Ci,r,E)$ are computed in Theorem \ref{casotras}  and  $\cquattro(E)$ is defined in Proposition \ref{confrontoaltezze}.
We finally remark that 
$$\hat{h}(P)=\hat{h}(P_0+Q)=\hat{h}(P_0)+\hat{h}(Q)\le (N-r) \hat{h}(\pi(P))+\hat{h}(Q)$$
where the last inequality is due to the fact that the $(N-t_\Ci+1)$-th coordinate of $P_0$ has maximal height.
Recall that by the  Zhang inequality $\hat{h}(Q)\le h(\Ci)/\deg \Ci$.
So
$$\hat{h}(P)\le \hat{h}(P_0)+\hat{h}(Q)\le (N-r) \hat{h}(\pi(P))+\frac{h(\Ci)}{\deg \Ci}.$$
Replacing (\ref{boundpiP}) we obtain the bound in the statement.

\end{proof}

\subsection{The Proof of Corollary \ref{caso_poly}}\label{esempiE2}
In \cite{EsMordell}, Theorem 6.2, we proved that the curve $\mathcal{C}$ is transverse in $E^2$ and its degree and height are bounded as
\[\deg\Ci=6n+9,\] \[h_2(\Ci)\leq 6(2n+3)\left(h_W(p)+\log m+2\cotto(E)\right)\]
where $h_W(p)=h_W(1:p_0:\ldots:p_n)$ is the height of the polynomial $p(X)$ and $\cotto(E)$ is given in Proposition \ref{confrontoaltezze}.
Applying Theorem \ref{casotras} to $\Ci$ in $E^2$ (with $N=2$ and $r=1$), we get that  
\begin{align*}
 \hat{h}(P)\leq& |D_K|^{5} 2^{41} 3^4 \left(h_2(\Ci)(\deg \Ci)+ \left(4\ccinque(E)+5.61\right)(\deg \Ci)^{2}\right)+4\cquattro(E).
\end{align*}
 Substituting the values of $\deg \Ci$ and $h_2(\Ci)$  we conclude the proof.
$\qedhere$

\subsection{The proof of Corollary \ref{cardinalita}}\label{dimocorcard}
To prove our corollary we use the same projection $\pi: E^N \to E^{t_\Ci}$ used in the  proof of Theorem \ref{thmgen}. Then
$\pi(\Ci)$ is transverse in  $E^{t_\Ci}$. In addition  $\deg \pi(\Ci)\le \deg \Ci$ and  $h(\pi(\Ci)) \le h(\Ci)$.    The proof of Theorem \ref{thmgen} also shows that there are no points on $\Ci$ of rank $<r_\Ci-t_\Ci$. We then apply    Theorem \ref{casotras}  and Theorem 1.3 of \cite{ioant}  to the curve $\pi(\Ci)$ in $E^{t_\Ci}$  to bound the number of points of rank $< t_\Ci$ on $\pi(\Ci)$. Note that  the rank of the projection of a point is less or equal than the rank of the point. Moreover the fiber of $\pi$ restricted to $\Ci$ has cardinality $\le \deg \Ci \deg E^{N-t_\Ci}$ by  B\'ezout theorem.}  This bounds the number of points on $\Ci$ of rank $<\max (r_\Ci-t_\Ci, t_\Ci)$ in terms of $\deg \Ci$, $\deg \pi(\Ci)$, $h_2(\pi(\Ci))$, $E$ and $N$. We finally use  $\deg \pi(\Ci)\le \deg \Ci$ and  $h(\pi(\Ci)) \le h(\Ci)$ to conclude.

\section*{Acknowledgments} I  thank Sara Checcoli for her contribution on many computations and for several remarks. I also thank \"Ozlem Imamoglu for some nice discussions.

\def\cprime{$'$}
\providecommand{\bysame}{\leavevmode\hbox to3em{\hrulefill}\thinspace}
\providecommand{\MR}{\relax\ifhmode\unskip\space\fi MR }
\providecommand{\MRhref}[2]{%
  \href{http://www.ams.org/mathscinet-getitem?mr=#1}{#2}
}
\providecommand{\href}[2]{#2}

\section*{}
%\noindent Sara Checcoli:
% Institut Fourier, 
%100 rue des Maths,
%BP74 38402 Saint-Martin-d'H\`eres Cedex, France.
%email: sara.checcoli@ujf-grenoble.fr
%\smallskip\\

Evelina Viada:
Mathematisches Institut,
Georg-August-Universit\"at,
Bunsenstra\ss e 3-5,
D-D-37073, G\"ottingen,
Germany.% and
%ETH Z\"urich,
%R\"amistrasse 101,
%8092, Zurich,
%Switzerland.

email: evelina.viada@math.ethz.ch

\end{document}